\documentclass[reqno,12pt]{amsart}

\usepackage[numbers]{natbib}  % 使用natbib增强引用功能

\usepackage{newtxtext}
\usepackage{mathptmx}
\usepackage{booktabs}
\numberwithin{equation}{section}
\DeclareMathAlphabet{\mathcal}{OMS}{cmsy}{m}{n}
\DeclareSymbolFont{largesymbols}{OMX}{cmex}{m}{n}
\DeclareUnicodeCharacter{02BC}{'}

\usepackage[unicode, colorlinks, linkcolor=blue, citecolor=red, urlcolor=teal, bookmarksnumbered]{hyperref}

\usepackage{pdfcomment}

\newcommand{\tref}[2]{\hyperref[#1]{\textsuperscript{\ref{#1}}\pdfmarkupcomment{\phantom{\ref{#1}}}{#2}}}

\usepackage{amsmath, amsthm, amssymb, bm, graphicx, mathrsfs}
\usepackage{cases}

\newtheorem{theorem}{Theorem}[section]
 
\newtheorem{proposition}[theorem]{Proposition}
\newtheorem{lemma}[theorem]{Lemma}

\newtheorem{corollary}[theorem]{Corollary}
\newtheorem{remark}[theorem]{Remark}

\begin{document}
	
	\title{The Logarithmic Laplacian on General Graphs}
	
	\author{Rui Chen}
	\address{School of Mathematical Sciences, Fudan University, Shanghai 200433, China}
	\email{chenrui23@m.fudan.edu.cn}
	
	\author{Wendi Xu}
	\address{Shanghai Institute for Mathematics and Interdisciplinary Sciences, Shanghai 200433, China}
	\email{xuwendi@simis.cn}
	
\begin{abstract}
	We establish, for the first time, a Bochner-type integral representation for the logarithmic Laplacian on weighted graphs. Assuming stochastic completeness of the underlying graph, we further derive an explicit pointwise formula for this operator:
{\scriptsize
	\[
	\log(-\Delta)\:u(x)
	=\frac{1}{\mu(x)}\sum_{y\neq x}W_{\log}(x,y)\,(u(x)-u(y))
	-\frac{1}{\mu(x)}\sum_{y}W(x,y)\,u(y)
	+\Gamma'(1)\,u(x).
	\]
}In the case of weighted lattice graphs with uniformly positive vertex measures, we obtain sharp two-sided bounds for the associated logarithmic kernel. Additionally, we prove that the logarithmic Laplacian is unbounded on $\ell^{2}$, and we present an alternative derivation of its pointwise form. Moreover, for every $1 < p \leq \infty$ and all $u \in C_c(\mathbb{Z}^{d})$, we establish a strong convergence in $\ell^{p}$:
	{\scriptsize\[\frac{(-\Delta)^{s} u - u}{s} \longrightarrow \log(-\Delta) \:u \quad \text{as } s \to 0^{+}.\]}Finally, on the standard lattice $\mathbb{Z}^{d}$, we compute the Fourier multipliers corresponding to both the fractional Laplacian and the logarithmic Laplacian, and derive exact large-time behavior and off-diagonal asymptotics of the associated diffusion kernels, including all sharp asymptotic constants.
\end{abstract}

\maketitle

\bigskip
\noindent{\normalfont\small\textbf{Keywords:} Fractional Laplacian, Logarithmic Laplacian, Lattice Graphs,  Diffusion Kernel}

\setcounter{tocdepth}{2}
\tableofcontents

\section{Introduction and Main Results}

In the past decade, there has been a surge of interest in nonlocal operators, both for their rich analytic structure and for wide-ranging applications in probability, geometry, and applied mathematics. Chief among these are the fractional Laplacian $\left(-\Delta\right)^s$ and its logarithmic counterpart $\log\left(-\Delta\right).$ On \(\mathbb{R}^d\), the fractional Laplacian has been intensively studied in \cite{caffarelli2007extension,di2012hitchhikerʼs,kwasnicki2017ten,chen2012dirichlet,ros2014dirichlet,stinga2010extension,servadei2012mountain,servadei2013variational,kwasnicki2012eigenvalues,feulefack2022small}. More recently, Chen and Weth \cite{chen2019dirichlet} were the first to introduce the logarithmic Laplacian in $\mathbb{R}^d$ as the derivative at \(s=0\) of \((-\Delta)^s\) and to derive its explicit integral representation. Since then, the logarithmic Laplacian has attracted considerable interest, resulting in  numerous applications and further developments in PDEs, spectral analysis and geometric perimeter (see, e.g., \cite{chen2023bounds,laptev2021spectral,chen2024cauchy,jarohs2020new,chen2023extension,zhang2021direct,hernandez2024optimal,de20190}).

More generally, there has been growing interest in defining and analyzing both the fractional Laplacian and the logarithmic Laplacian on Riemannian manifolds, where curvature and volume growth introduce novel geometric challenges. To define the fractional Laplacian, two main approaches are the Caffarelli–Silvestre extension and the heat semigroup method via functional calculus. A number of works have begun to develop a systematic theory in this setting, covering both compact manifolds \cite{d2016fractional,feizmohammadi2024fractional,li2024inverse,de2017fractional} and noncompact manifolds \cite{banica2015some,caselli2023asymptotics,caselli2024fractional,bhowmik2022extension,papageorgiou2024asymptotic,alonso2018integral,papageorgiou2024large,kim2022harnack,anker2024schr}.

In contrast, there have been no results on the logarithmic Laplacian on manifolds until very recently.  To the best of our knowledge, only two papers have addressed the related topics: in \cite{pramanik2025anisotropic},  \(\log(-\Delta + m\,\mathrm{I})\) for \(m>1\) was defined on closed manifolds, and recently in \cite{chen2025logarithmic}, we have just given a unified definition of \(\log(-\Delta)\) on arbitrary complete manifolds.

The study of geometric and analytic properties of graphs has a long tradition \cite{chung1997spectral,chung1999coverings,bauer2017sharp,wojciechowski2009heat,wojciechowski2008stochastic,horn2019volume,bauer2015li}. Lately, increasing studies have focus on partial differential equations on graphs: the discrete Schrödinger equation and its variants have been explored in \cite{hua2023existence,yang2024normalized,grigor2017existence,zhang2018convergence}; the fractional Schrödinger equations on graphs have been studied in \cite{zhang2024fractional1,zhang2024fractional,ciaurri2018nonlocal} and the logarithmically nonlinear Schrödinger equations in \cite{chang2023convergence,he2024existence,ardila2017existence,chang2023ground}.  Despite these advances, a notion of the logarithmic Laplacian itself has not yet been developed in the discrete setting.  

 On a finite, connected, weighted graph \(G=(V,E,\mu,w)\) with $N$ vertices, the normalized graph Laplacian \(-\Delta\) is a finite‐dimensional, self-adjoint matrix with eigenvalues $0=\lambda_0<\lambda_1\le\cdots\le\lambda_{N-1}\,,$ 
and the corresponding orthonormal eigenfunctions \(\{\varphi_j\}\).  By functional calculus, it is nature to define $$(-\Delta)^s u
=\sum_{j=0}^{N-1}\lambda_j^s\langle u,\varphi_j\rangle \varphi_j,\quad\log(-\Delta)u
=\sum_{j=1}^{N-1} \log(\lambda_j)\,\langle u,\varphi_j\rangle\,\varphi_j,$$
where $\langle \cdot,\cdot\rangle$ represents the inner product in $\ell^2\left(V,\mu\right)$.

On an infinite graph, such a global spectral decomposition is generally unavailable.  Instead, one uses the Bochner‐integral representation for the fractional Laplacian:
\[
(-\Delta)^s u
=\frac{s}{\Gamma(1-s)}\int_{0}^{\infty}\bigl(u-e^{t\Delta}u \bigr)\,t^{-1-s}\,dt,
\] 
which converges for \(u\) in suitable domains whenever \(G\) is stochastically complete.  Under this hypothesis, one obtains the pointwise kernel formula \cite{zhang2024fractional1,wang2023eigenvalue}:
\[
(-\Delta)^s u(x)
=\sum_{y\in V}W_s(x,y)\,\bigl(u(x)-u(y)\bigr),
\]
where
\[
W_s(x,y)
=\frac{s}{\Gamma(1-s)}\int_{0}^{\infty}p(t,x,y)\,t^{-1-s}\,dt,
\]
and \(p(t,x,y)\) is the continuous‐time heat kernel.  

A natural question is how to generalize the logarithmic Laplacian to infinite graphs. To the best of our knowledge, no one has yet defined the logarithmic Laplacian on infinite graphs, and correspondingly, there is no prior work on its properties in graph setting.  In this paper, we fill this gap by introducing, for the first time, a discrete Bochner‐integral formula for the logarithmic Laplacian, see Theorem \ref{bochlog}:
\[
	\log(-\Delta)\,u
=\int_{0}^{\infty}\frac{e^{-t}u - e^{t\Delta}u}{t}\,dt,
\]
which is well‐defined on any weighted graph (under mild growth conditions on \(u\)).  Moreover, on stochastically complete graphs, this definition yields a pointwise representation. Since we do not impose any global curvature or volume‐growth assumptions that would guarantee long‐time heat‐kernel decay, we first assume the following integrability condition:
\begin{equation}\label{integr}
	\int_{1}^{\infty}\frac{p(t,x,y)}{t}\,dt<\infty
	\quad\text{for all }x,y\in V.
\end{equation}
In later sections, we will show that this hypothesis is satisfied in a wide class of graphs.

\begin{theorem}\label{pointlog11}
	Let \(G=(V,E,\mu,w)\) be an infinite, connected, stochastically complete weighted graph, and $p(t,x,y)$ be the heat kernel of $G$ satisfying (\ref{integr}). For \(u\in C_c(V)\), the logarithmic Laplacian can be written pointwise as
	\begingroup\small	\[
	\bigl(\log(-\Delta)u\bigr)(x)
	=\frac{1}{\mu(x)}\sum_{\substack{y\in V,\:y\neq x}}
	W_{\log}(x,y)\bigl(u(x)-u(y)\bigr)
	\;-\;\frac{1}{\mu(x)}\sum_{y\in V}W(x,y)\,u(y)
	\;+\;\Gamma'(1)\,u(x),
	\]
	\endgroup
	where the positive, symmetric kernels are given by
	\[
	W_{\log}(x,y)
	:=\mu(x)\,\mu(y)\int_{0}^{1}\,\frac{p(t,x,y)}{t}dt,
	\quad
	W(x,y)
	:=\mu(x)\,\mu(y)\int_{1}^{\infty}\frac{p(t,x,y)}{t}dt.
	\]
\end{theorem}

These novel representations lay the groundwork for the analysis of differential and nonlocal equations invloving the logarithmic Laplacian, enable a detailed spectral study of the operator, and open the door to a wide range of futher applications on graphs.

To carry out a more refined analysis of the logarithmic Laplacian on graphs, we need sharp estimates on its kernel.  This in turn requires precise control of the heat kernel under additional structural assumptions on the graph.  In what follows, we focus on weighted lattice graphs \(\mathbb{Z}^d=(V,E,\mu,w)\) satisfying
\[
0<\mu_{\min}:=\inf_{x\in\mathbb{Z}^d}\mu(x)\le 
\mu_{\max}:=\sup_{x\in\mathbb{Z}^d}\mu(x)<\infty,
\]
equipped with the normalized Laplacian.  Under these hypotheses, one has the following heat‐kernel bounds (see subsection \ref{rehe}):
\begin{equation}\label{gaussian}
	\frac{C'}{V\bigl(x,\sqrt{t}\bigr)}
	\exp\!\biggl(-c'\,\frac{d(x,y)^2}{t}\biggr)
	\;\le\;
	p(t,x,y)
	\;\le\;
	\frac{C}{V\bigl(x,\sqrt{t}\bigr)},\quad  t>0,
\end{equation}
together with the refined upper bound estimates:
\begin{enumerate}
	\item If \(t \ge d(x,y)>0\), then
	\begin{equation}\label{222}
		p(t,x,y)\;\le\;\frac{C}{V(x,\sqrt{t})}\exp\!\Bigl(-C'\,\frac{d(x,y)^2}{t}\Bigr).
	\end{equation}
	\item If \(d(x,y)>t>0\), then
	\begin{equation}\label{444}
		p(t,x,y)\;\le\;\frac{C}{V(x,\sqrt{t})}\,e^{d(x,y)}\Bigl(\frac{t}{d(x,y)}\Bigr)^{d(x,y)}.
	\end{equation}
\end{enumerate}
In particular, for \(x\ne y,\:0 < t \le \delta,\:\delta<1\),
\begin{equation}\label{333}
	p(t,x,y)\;\le\;\frac{C}{d(x,y)!}.
\end{equation}

Using these bounds, we establish the following two propositions.

\begin{proposition}\label{prop:Wlog11}
	The following identity hold:
	\begin{equation}\label{wlog}
		\sum_{y\ne x}W_{\log}(x,y)
		=\mu(x)\int_0^1\bigl(1-p(t,x,x)\mu(x)\bigr)\,\frac{dt}{t}.
	\end{equation}
	Moreover, we have the following estimate:
	\begin{equation}\label{wlogg}
		W_{\log}(x,y)\lesssim \frac{e^{d(x,y)}}{d(x,y)^{d(x,y)+1}},\quad x\ne y.
	\end{equation}
\end{proposition}

\begin{proposition}\label{prop:Wxy11}
	The following estimates hold:
	\[
	d(x,y)^{-d}\lesssim
	W(x,y)\lesssim d(x,y)^{-d},
	\quad x\ne y,
	\]
	and $$	W(x,x)\sim 1.$$
\end{proposition}

We recall that the spectrum of the normalized graph Laplacian lies in \([0,2]\).  Consequently, by functional calculus, the multiplier \(\log\lambda\) is unbounded in \((0,2]\), and thus \(\log(-\Delta)\)  is unbounded operator on \(\ell^2(V,\mu)\).  This unboundedness is also apparent from the pointwise kernel representation of \(\log(-\Delta)\).  Indeed, the first term induces a bounded gradient‐type operator on \(\ell^2(V,\mu)\), whereas the second, long‐range term fails to be \(\ell^2\)-bounded.  We make this precise in the next two propositions.

\begin{proposition}\label{prop:unbounded22}
	There exists a constant \(C>0\), independent of \(u\), such that
	\[
	\bigl\|\nabla^{\log}u\bigr\|_{\ell^2(\mathbb{Z}^d)}
	\;\le\;
	C\,\|u\|_{\ell^2(\mathbb{Z}^d)},
	\quad\forall\,u\in\ell^2(\mathbb{Z}^d),
	\]
	where
\[
\nabla^{\log}u(x)
=\Bigl(
\sqrt{\tfrac{W_{\log}(x,y_i)}{2\,\mu(x)}}\bigl(u(x)-u(y_i)\bigr)
\Bigr)_{i\ge1}.
\]

\end{proposition}

\begin{proposition}\label{prop:unbounded11}
	There exists a sequence \(\{u_n\}_{n\ge1}\subset\ell^2(\mathbb{Z}^d)\) with $\|u_n\|_{\ell^2(\mathbb{Z}^d)}=1$
	such that
	\[
	\sum_{x,y\in\mathbb{Z}^d}u_n(x)\,u_n(y)\,W(x,y)
	\;\longrightarrow\;+\infty
	\quad\text{as }n\to\infty,
	\]
	where
	\[
	W(x,y)
	=\mu(x)\,\mu(y)\int_1^\infty p(t,x,y)\,\frac{dt}{t}
	\]
	is the long‐range kernel appearing in the logarithmic Laplacian.
\end{proposition}

Below, we exploit the pointwise integral representations of the fractional and logarithmic Laplacians to derive two convergence results.  In many cases, these pointwise arguments yield stronger convergence than those obtained via functional calculus, while imposing significantly weaker regularity assumptions on the underlying functions.

\begin{proposition}\label{slx22211}
	For every \(u \in \ell^\infty(\mathbb{Z}^d)\),
	\[
	\lim_{s\to 1^-} \bigl\| (-\Delta)^s u + \Delta u \bigr\|_{\ell^\infty(\mathbb{Z}^d)} = 0.
	\]
\end{proposition}

\begin{proposition}\label{slx11111}
	For every \(u\in C_c(\mathbb{Z}^d)\), one has pointwise convergence
	\[
	\lim_{s\to0^+}(-\Delta)^s u(x) = u(x),
	\quad x\in \mathbb{Z}^d.
	\]
\end{proposition}

We also show how the pointwise representation leads to an alternative proof of Theorem~\ref{pointlog11} on $\mathbb{Z}^d;$ see Theorem \ref{pointwise}.

We further establish that the derivative‐at‐zero limit defining the fractional Laplacian actually converges in every $\ell^p$ space for $1<p\le\infty$, and moreover, the logarithmic Laplacian of any compactly supported function belongs to $\ell^p$.

\begin{theorem}\label{thm:lplimit11}
	Let \(u\in C_c(\mathbb{Z}^d)\).  Then, for every \(1<p\le\infty\), one has
	\[
\log\left(-\Delta\right)u
	\;\in\;\ell^p(\mathbb{Z}^d),\]
	and the strong convergence in $\ell^p\left(\mathbb{Z}^d\right)$
	\[\frac{(-\Delta)^s u - u}{s}\longrightarrow
\log\left(-\Delta\right)u
	\quad\text{as }s\to0^+.
	\]
\end{theorem}

Finally, we restrict to the standard lattice graph \(\mathbb{Z}^d\) (with unit weights and counting measure) and employ the discrete Fourier transform
\[
\widehat u(\xi)=\frac1{(2\pi)^{d/2}}\sum_{x\in\mathbb{Z}^d}u(x)e^{-i x\cdot\xi}, 
\quad \xi\in[-\pi,\pi]^d.
\]
In this framework, one shows that
\[
\widehat{(-\Delta)^su}(\xi)=\Phi(\xi)^s\widehat u(\xi),
\qquad
\widehat{\log(-\Delta)\,u}(\xi)=\ln\bigl(\Phi(\xi)\bigr)\,\widehat u(\xi),
\]
where \(\Phi(\xi):=\sum_{j=1}^d(2-2\cos\xi_j)\) is the graph symbol.  Consequently the fractional diffusion kernel admits the Fourier integral representation:
\[
p_s(t,x,y)
=\bigl(e^{-t(-\Delta)^s}\delta_y\bigr)(x)
=\frac1{(2\pi)^d}
\int_{[-\pi,\pi]^d}e^{-t\,\Phi(\xi)^s}\,e^{i(x-y)\cdot\xi}\,d\xi.
\]
It follows immediately that \(p_s(t,x,y)\to\delta_{x,y}\) as \(t\to0\).  We then analyze the exact large-time behavior and off-diagonal asymptotics of fractional diffusion kernel.

On \(\mathbb{R}^d\), the classical fractional heat kernel satisfies
\[
\mathcal P_s(t,x)
=\frac{1}{(2\pi)^d}\int_{\mathbb{R}^d}e^{-t|\xi|^{2s}}e^{i x\cdot\xi}\,d\xi
=t^{-\frac{d}{2s}}
\frac{1}{(2\pi)^d}
\int_{\mathbb{R}^d}e^{-|\xi|^{2s}}\,e^{i\,x\cdot \xi\,t^{-\frac{1}{2s}}}\,d\xi.
\]
Therefore, one has the sharp time decay and asymptotic constants 
\[
\sup_{x\in \mathbb{R}^d}\mathcal P_s(t,x)\;\lesssim\;t^{-\frac{d}{2s}},
\quad
\lim_{t\to\infty}t^{\frac{d}{2s}}\,
\mathcal P_s(t,x)
=\frac{1}{(2\pi)^d}\int_{\mathbb{R}^d}e^{-|\xi|^{2s}}\,d\xi
=\frac{\pi^{-d/2}}{s\,2^d\,\Gamma(\tfrac d2)}\,
\Gamma\!\Bigl(\tfrac{d}{2s}\Bigr),
\]
and in \cite[Theorem 2.1]{blumenthal1960some}, 
\[\mathcal P_s(t,x) \sim |x|^{-d-2s}\quad\text{as}\quad |x|\rightarrow\infty.\]

Next, we prove that exactly the same temporal decay rate and spatial off‐diagonal behavior also hold for the fractional diffusion kernel \(p_s(t,x,y)\) on \(\mathbb{Z}^d\). In fact, the large–time asymptotic analysis of $p_s(t,x,y)$ on $\mathbb{Z}^d$ is essentially the same as on $\mathbb{R}^d$; both the decay rate and the limiting constant $C_{s,d}$ follow by identical methods.
By contrast, the spatial–frequency decay analysis is fundamentally different, see Proposition \ref{prop:large_time11} and \ref{prop:tail_asymptotic11}.

\begin{proposition}\label{prop:large_time11}
	For each fixed \(0<s<1, d\ge 1\), the fractional diffusion kernel satisfies
	\[
	\sup_{x,y\in\mathbb{Z}^d}p_s(t,x,y)\;\lesssim\;t^{-d/(2s)},\quad t>0,
	\]
	and moreover,
	\[
	\lim_{t\to\infty}t^{\frac{d}{2s}}\,p_s(t,x,y)
	= C_{s,d},
	\]
	where 
	$$
	C_{s,d}
	:=\frac{\pi^{-d/2}}{s\,2^d\,\Gamma(\tfrac d2)}\,
	\Gamma\!\bigl(\tfrac{d}{2s}\bigr).
	$$
\end{proposition}

\begin{proposition}\label{prop:tail_asymptotic11}
	Fix \(0<s<1\) and \(t>0,\:d\ge 1\).  As \(\lvert x-y\rvert\to\infty\), the discrete fractional diffusion kernel on \(\mathbb{Z}^d\) satisfies
	\[
	\lim\limits_{\lvert x-y\rvert\rightarrow \infty}\lvert x-y\rvert^{\,d+2s}p_s(t,x,y)=-\frac{t}{\left(2\pi\right)^d}A_{s,d},
	\]
	where $A_{s,d}$ is given by Lemma \ref{jieduan11}.
\end{proposition}

To the best of our knowledge, this is the first time the precise off‐diagonal decay 
\(\lvert x-y\rvert^{-d-2s}\) on \(\mathbb{Z}^d\) has been derived by purely analytic methods.  
A closely related result was previously obtained from a probabilistic viewpoint in \cite[Theorem 2.4]{bendikov2013alpha}.

	In $\mathbb{R}^d$, the radial symmetry of the kernel allows one to pass immediately to spherical coordinates and then apply Bessel‐function identities to obtain the off–diagonal asymptotics (see \cite[Theorem 2.1]{blumenthal1960some}).  On torus $\mathbb{T}^d$, however, no such polar change of variables is available.  Our first key idea is to insert a smooth cutoff and split the Fourier integral into an interior
	 region (a small ball) and its complement. On the complement, classical estimates yield arbitrarily high order decay (\cite[Theorem 3.2.9]{grafakos2008classical}), while the interior contribution is handled in the following crucial Lemma \ref{jieduan11}.
	 
	  The proof of Lemma \ref{jieduan11} relies on classical properties of Bessel functions, combined with delicate convergence estimates. The full technical details are carried out in Appendix  for the interested reader. Since the untruncated integral in (\ref{wjdss}) is not well-defined, the cutoff function plays an essential role in making the argument rigorous.

\begin{lemma}\label{jieduan11}
	Let \(d\ge1,\) \(-\frac{d}{2}<s<1\) and  \(\omega\in \mathbb{S}^{d-1}\) be an arbitrary fixed unit vector.  Suppose radial function 
	\(\chi\in C_c^\infty(\mathbb{R}^d)\) satisfies
	\[
	\chi(\eta)=
	\begin{cases}
		1, & |\eta|\le \frac{1}{2},\\
		0, & |\eta|\ge 1.
	\end{cases}
	\]
	Then the limit
	\begin{equation}\label{wjdss}
	A_{s,d}
		:=\lim_{N\to\infty}
		\int_{\mathbb{R}^d}|\eta|^{2s}\,\chi\!\Bigl(\frac{\eta}{N}\Bigr)\,e^{\,i\,\omega\cdot\eta}\,d\eta
	\end{equation}
	exists, is finite, and depends only on \(d\) and \(s\), not on the particular choice of \(\chi\) or on \(\omega\).
\end{lemma}

Similarly, the log‐diffusion kernel is given by
\[
p_{\log}(t,x,y)
=\bigl(e^{-t\log(-\Delta)}\delta_y\bigr)(x)
=\frac1{(2\pi)^d}
\int_{[-\pi,\pi]^d}\Phi(\xi)^{-t}\,e^{\,i\,(x-y)\cdot\xi}\,d\xi,
\]
and is well‐defined exactly for $0\le t<\frac d2.$ In particular,
\[
p_{\log}(0,x,y)=\delta_{x,y},
\qquad
p_{\log}(t,x,x)
=\frac1{(2\pi)^d}\int_{[-\pi,\pi]^d}\Phi(\xi)^{-t}\,d\xi
<\infty.
\]

On $\mathbb{R}^d$, the corresponding log‐diffusion kernel \cite{chen2024cauchy}
$$\displaystyle \mathcal P_{\log}(t,x) =\frac1{(2\pi)^d}\int_{\mathbb{R}^d}|\xi|^{-2t}e^{i x\cdot\xi}\,d\xi$$
exhibits a pole of order $(d-2t)^{-1}$ as $t\to\frac d2$ and decays like $\lvert x\rvert^{\,2t-d}$ for large $\lvert x\rvert$. We now analyze its blow‐up and off‐diagonal decay on $\mathbb{Z}^d.$

\begin{proposition}\label{prop:log_blowup_and_delta11}
	Let \(d\ge 1\). Then, for every fixed \(x\ne y\in \mathbb{Z}^d\), we have
	\[
	\lim\limits_{t\rightarrow \frac{d}{2}}\left(d-2t\right)p_{\log}(t,x,y)=|\mathbb{S}^{d-1}|,
	\]
	where $\mathbb{S}^{d-1}$ is the unit sphere in $\mathbb{R}^d.$
\end{proposition}

	On \(\mathbb{Z}^d\), the proof of time blow-up behavior as \(t \to \frac{d}{2}\) is fundamentally different from the continuous case in \(\mathbb{R}^d\). In \(\mathbb{R}^d\), one can directly exploit the explicit formula for the Fourier transform of radial functions to extract the asymptotics. However, on the discrete lattice \(\mathbb{Z}^d\), such exact expressions are not available. Instead, our approach relies on a time-rescaling argument, where the main idea is to recognize that the dominant contribution to the asymptotic behavior comes from the region where \(\xi\) is small. This allows us to isolate the leading singularity as \(t \to \frac{d}{2}\) by focusing on the low-frequency behavior of the symbol.

\begin{proposition}\label{prop:log_offdiag_asym11}
	Fix \(d\ge1\) and \(t\in(0,\tfrac d2)\).  As \(\lvert x-y\rvert\to\infty\), the discrete log‐diffusion kernel satisfies
	\[
	\lim\limits_{\lvert x-y\rvert\rightarrow \infty}\lvert x-y\rvert^{\,d-2t}p_{\log}(t,x,y)=\frac{A_{-t,d}}{\left(2\pi\right)^d},
	\]
		where $A_{-t,d}$ is given by Lemma \ref{jieduan11}.
\end{proposition}

 The argument follows the fractional heat kernel case almost identically, employing the same smooth cutoff construction and relying critically on Lemma \ref{jieduan11}.

The large-time behavior and off-diagonal asymptotics of the fractional and logarithmic diffusion kernels on $\mathbb{R}^d$ and $\mathbb{Z}^d$ are clearly summarized in the two tables below.

\begin{table}[ht]
	\centering
	\caption{Comparison of \(\mathcal{P}_s\) and \(\mathcal{P}_{\log}\) on \(\mathbb{R}^d\)}
	\label{tab:Rn-kernels}
	\begin{tabular}{lcccc}
		\toprule
		Kernel & Lifespan & Time asymptotics & Behavior at \(x=0\) & Decay as \(\lvert x\rvert\to\infty\) \\
		\midrule
		\(\mathcal{P}_s(t,x)\)        & \((0,\infty)\)              & \(t^{-d/(2s)}\) decay at infty      & smooth        & \(\lvert x\rvert^{-d-2s}\)   \\
		\(\mathcal{P}_{\log}(t,x)\)    & \(\bigl(0,\tfrac d2\bigr)\) & blow-up rate \((d-2t)^{-1}\) & singular  $|x|^{2t-d} $   & \(\lvert x\rvert^{2t-d}\)    \\
		\bottomrule
	\end{tabular}
\end{table}

\begin{table}[ht]
	\centering
	\caption{Comparison of \(p_s\) and \(p_{\log}\) on \(\mathbb{Z}^d\)}
	\label{tab:Zd-kernels}
	\begin{tabular}{lcccc}
		\toprule
		Kernel & Lifespan & Time asymptotics & Decay as \(\lvert x-y\rvert\to\infty\) \\
		\midrule
		\(p_s(t,x,y)\)     & \((0,\infty)\)              & \(t^{-d/(2s)}\) decay at infty    & \(\lvert x-y\rvert^{-d-2s}\)  \\
		\(p_{\log}(t,x,y)\) & \(\bigl(0,\tfrac d2\bigr)\) & blow-up rate \((d-2t)^{-1}\)                      & \(\lvert x-y\rvert^{2t-d}\)   \\
		\bottomrule
	\end{tabular}
\end{table}

The remainder of the paper is organized as follows.  In Section 2, we provide a brief review of the foundational setting of weighted graphs, including Dirichlet and Neumann forms, as well as the functional calculus for graph Laplacians. Section 3 is devoted to the Bochner‐integral definition of the logarithmic Laplacian on infinite graphs.  In Section 4, we consider the case of weighted lattice graphs and obtain sharp asymptotics for the associated logarithmic kernel, which are then used to prove several convergence and integrability properties of the logarithmic Laplacian. Finally, we turn to the Fourier‐analytic perspective on $\mathbb{Z}^d$, compute the symbols of both the fractional Laplacian and the logarithmic Laplacian, analyze their large-time and off-diagonal asymptotic behavior, and derive the precise asymptotic coefficients. The complete proof of Lemma \ref{jieduan11} is provided in Section 5.

Throughout this paper, \(C\) denotes a generic positive constant, whose value may change from line to line and does not depend on the main parameters.  We also use the symbols
\[
A \lesssim B
\quad\text{and}\quad
A \gtrsim B
\]
to indicate
\[
A \le C\,B
\quad\text{and}\quad
A \ge \frac1C\,B
\]
for constant \(C>0\), respectively. In particular, \(A\sim B\) means both \(A\lesssim B\) and \(A\gtrsim B\) hold.

\section{Background and Preliminaries}

	\subsection{Foundations of Graph Theory}

	We begin by recalling some basic notions related to weighted graphs. Let \( V \) be a countable set of vertices, and let \( E = \{ \{x, y\} \subset V : x \sim y \} \) be the set of undirected edges, where \( x \sim y \) means that the vertices \( x \) and \( y \) are adjacent. A graph is called locally finite if for all \( x \in V \), the set \( \{ y \in V : x \sim y \} \) is finite; it is called connected if any two vertices can be joined by a finite path. 
	
	Each vertex \( x \in V \) is assigned a measure \( \mu(x) > 0 \), each edge \( \{x, y\} \in E \) is assigned a symmetric weight \( w_{xy} = w_{yx} > 0 \), and if  \( \{x, y\} \notin E \), $w_{xy}=0.$ The quadruple \( G = (V, E, \mu, w) \) is referred to as a weighted graph. The degree of a vertex \( x \in V \) is defined by $	m(x) := \sum_{y \sim x} w_{xy}.$
	The graph distance \( d(x, y) \) between two vertices \( x, y \in V \) is defined as the minimal number of edges in any path connecting \( x \) and \( y \). For a vertex $x \in V$ and radius $r \ge 0$, we define the (closed) ball centered at $x$ by
	\[
	B(x, r) := \{ y \in V : d(x, y) \le r \}.
	\]
	Given a subset $A \subset V$, we define its volume by $V(A) := \sum_{x \in A} \mu(x).$
	In particular, we write $V(x, r)$ to denote the volume of the ball $B(x, r)$, that is, $	V(x, r) := V(B(x, r)).$
	Moreover, a weighted graph is called stochastically complete \cite{wojciechowski2008stochastic,hua2017stochastic,folz2014volume} if 
	\[\sum_{y\in V}p\left(t,x,y\right)\mu\left(y\right)=1,\:\:\forall t\ge 0,\:x\in V.\]
	
The requirement that graphs be stochastically complete mirrors the corresponding condition in the classical continuous setting. This reflects the fact that the standard heat equation does not include any absorption or killing terms. From the probabilistic point of view, this corresponds to the total probability being preserved under Brownian motion, which is the diffusion process generated by the Laplacian.

	Let \( C(V) \) denote the set of all real-valued functions on \( V \), and \( C_c(V) \subset C(V) \) denote the subspace of functions with finite support. For a function \( u \in C(V) \), its integral over the graph is given by $	\int_V u\, d\mu := \sum_{x \in V} u(x)\mu(x).$
	For \( p \in [1, \infty) \), we define the space \( \ell^p(V) \) of functions with finite \( p \)-norm:
	\[
	\|u\|_{\ell^p(V)} := \left( \sum_{x \in V} |u(x)|^p \mu(x) \right)^{1/p}.
	\]
	The space \( \ell^\infty(V) \) consists of all bounded functions on \( V \), with norm $\|u\|_{\ell^\infty(V)} := \sup_{x \in V} |u(x)|.$

Next, we consider that $G = (V, E, \mu, w)$ is a weighted graph, where $\mu$ is a positive measure, and $w : V \times V \to [0,\infty)$ is a symmetric edge function satisfying $w_{xy} = w_{yx},w_{xx} = 0$ and 
\[\sum_{y\in V}w_{xy}<\infty \quad \text{for all}\:\:x\in V.\]
In particular, the locally finite graph satisfies this condition.

Define the subspace of $C\left(V\right)$ given by
\[\mathcal{D}=\left\{ u \in C\left(V\right) \;\middle|\; \sum_{x,y \in V} w_{xy} \bigl(u(x) - u(y)\bigr)^2 < \infty \right\},\]
and the bilinear map
\[
\mathcal{Q}(u,v) := \frac{1}{2} \sum_{x,y \in V} w_{xy} \bigl(u(x) - u(y)\bigr)\bigl(v(x) - v(y)\bigr),
\quad u,v \in \mathcal{D}.
\]
We consider the energy norm
\[
\|u\|_Q := \left( \mathcal{Q}(u) + \|u\|_{\ell^2(V,\mu)}^2 \right)^{1/2},
\]
and define the form $Q^{\left(N\right)}$ as the restriction of $Q$ to
\[
\operatorname{Dom}(Q^{(N)}) := \ell^2(V,\mu)\cap \mathcal{D}.
\]
We refer to $Q^{(N)}$ as the Neumann form, as it corresponds to the maximal closed extension of the energy form without imposing vanishing conditions at the boundary. The space $\left(\operatorname{Dom}(Q^{(N)}),\|\cdot\|_Q\right)$ becomes a Hilbert space.

By the theory of closed forms, there exists a unique positive operator $-L^{(N)}$ such that
\[
Q^{(N)}(u,v) = \langle -L^{(N)} u, v \rangle_{\ell^2(V,\mu)},
\quad \text{for all } u \in \operatorname{Dom}(-L^{(N)}),\; v \in \operatorname{Dom}(Q^{(N)}).
\]
In this paper, $A$ is called positive operator if $A$ is self-adjoint and $\sigma(A)\subset [0,\infty).$ The operator $L^{(N)}$ is called the Neumann Laplacian.

To impose Dirichlet-type conditions, we consider the subspace \( C_c(V) \subset \ell^2(V,\mu) \) of finitely supported functions. The Dirichlet form \( Q^{(D)} \) is defined as the closure of the restriction of \( Q^{(N)} \) to \( C_c(V) \) with respect to the norm $\|\cdot\|_Q$. That is,
\[
\operatorname{Dom}(Q^{(D)}) := \overline{C_c(V)}^{\|\cdot\|_Q}, \qquad
Q^{(D)} := Q^{(N)}\big|_{\operatorname{Dom}(Q^{(D)})}.
\]
The corresponding unique self-adjoint operator \( L^{(D)} \), called the Dirichlet Laplacian with $\sigma(-L^{(D)}) \subset [0,\infty).$ If \( C_c(V) \) is dense in \( \operatorname{Dom}(Q^{(N)}) \) with respect to the norm \( \|\cdot\|_Q \), then \( Q^{(N)} = Q^{(D)} \), the Neumann and Dirichlet Laplacians coincide and we simply write
\[
Q := Q^{(N)} = Q^{(D)}, \quad
L := L^{(N)} = L^{(D)}.
\]

In general, the domain of the Dirichlet Laplacian $L^{(D)}$ is difficult to describe explicitly. Nevertheless, the pointwise action of the operator is easily understood through its associated formal expression. To this end, we introduce the formal Laplace operator $\Delta$. It acts on functions $u : V \to \mathbb{R}$ via
\[
\Delta \:u(x) := \frac{1}{\mu(x)} \sum_{y \in V} w_{xy} \bigl( u(y) - u(x) \bigr), \quad x \in V,
\]
well defined in
\[
\mathcal{F}=\Bigl\{\,u\in C(V)\;\Bigm|\;\sum_{y\in V}w_{xy}\,\lvert u(y)\rvert<\infty, \text{for all }x\in V\Bigr\} .
\]
It is obvious that $\operatorname{Dom}(Q^{\left(N\right)}) \subset \mathcal{D}\subset \mathcal{F}.$

This formal Laplacian can be used to identify the $L^{\left(D\right)}$ (see \cite[Theorem~1.6]{keller2021graphs}), and serves as the discrete analogue of the classical Laplace operator in Euclidean space and plays a fundamental role in analysis on graphs. Two special cases are of particular interest. The first is when the measure is uniform, i.e., \(\mu \equiv 1\), which corresponds to the so-called standard graph Laplacian. The second is when the measure \(\mu(x)  = m(x)\), in which case the resulting operator is the normalized Laplacian.

In this subsection we address the uniqueness of self‐adjoint extensions of the formal Laplacian $\Delta$ when restricted to the space of finitely supported functions.

Recall that a symmetric operator on a dense domain in a Hilbert space is said to be essentially self‐adjoint if it admits exactly unique self‐adjoint extension.

Let
\[
\Delta_{\min} := \Delta\big|_{\operatorname{Dom}(\Delta_{\min})}, 
\qquad
\operatorname{Dom}(\Delta_{\min}) = C_c(V),
\]
provided $\Delta \:C_c(V)\subset \ell^2(V,\mu)$.  \cite[Theorem~1.29]{keller2021graphs} characterizes this condition. In particular, this condition is satisfied if the graph is locally finite or if $\mu_{min}:=\inf_{x\in V} \mu(x)>0.$

By Green’s formula (\cite[Proposition 1.5]{keller2021graphs}), $\Delta_{\min}$ is symmetric. 
Under these hypotheses, $L^{(D)}$ provides at least one self‐adjoint extension of $\Delta_{\min}$, and the question of essential self‐adjointness reduces to whether any other self‐adjoint extensions exist.

From \cite{keller2021graphs}, we know that
\[\operatorname{Dom}\left(-\Delta_{min}\right)\subset \operatorname{Dom}\left(-L^{\left(D\right)}\right)\subset \operatorname{Dom}\left(\left(-\Delta_{min}\right)^{*}\right)\subset \mathcal{F}.\]
By general operator theory, $\Delta_{\min}$ is essentially self‐adjoint if and only if $\left(-\Delta_{min}\right)^*$ is self‐adjoint if and only if $\operatorname{Dom}\left(L^{\left(D\right)}\right)= \operatorname{Dom}\left(\left(-\Delta_{min}\right)^{*}\right)$. Thus, if $\Delta_{min}$ is essentially self‐adjoint, then $Q^{\left(D\right)}=Q^{\left(N\right)},$ so $-L^{\left(D\right)}=-L^{\left(N\right)}=\left(-\Delta_{min}\right)^{*}$ is unique self‐adjoint extension.

The following two conditions can ensure $\Delta_{min}$ is essentially self‐adjoint \cite[Theorem~6]{keller2012dirichlet}.

\begin{description}
	\item[($C_1$)] Uniform positivity of the measure: $\mu_{min}:=\inf_{x\in V}\mu(x)>0.$
\end{description}
We also consider a weaker, path‐based growth requirement:

\begin{description}
	\item[$\left(C_2\right)$] Infinite paths carry infinite measure:
	For every sequence of distinct vertices $\left\{x_n\right\}_{n\ge1}$ satisfying $x_n\sim x_{n+1}$ for all $n$, one has
	\[
	\sum_{n=1}^\infty \mu(x_n) =\infty.
	\]
\end{description}
Clearly, $(C_1)$ implies $\left(C_2\right).$

Throughout the remainder of this paper, we assume that $\Delta\big|_{C_c\left(V\right)}$ is essentially self‐adjoint.  In that case its unique self‐adjoint extension satisfies
\[
-\tilde{\Delta}:=-L^{(D)}=-L^{(N)}=\bigl(-\Delta_{\min}\bigr)^*.
\]
For the sake of simplicity, we will henceforth write \(-\Delta\) in place of \(-\tilde{\Delta}\), which is a positive operator with domian
\[
\operatorname{Dom}(-\Delta)
\;=\;\bigl\{\,u\in\ell^2(V,\mu)\;\big|\;-\Delta u\in \ell^2(V,\mu)\bigr\}.
\]

In particular, if connected graph is locally finite and satisfies condition \((C_2)\), or more generally if it satisfies condition \((C_1)\), then $\Delta_{min}$ is essentially self‐adjoint.

\subsection{Framework of Functional Calculus}\label{sec:spectral}

In this subsection, we begin to develop a functional calculus framework for Laplace operator on Graphs. Specifically, we restrict our attention to several prototypical positive operators, such as
\begin{itemize}
	\item \textbf{The (combinatorial) Laplacian on a finite graph.}  
	Here $-\Delta$ acts on $\ell^2(V,\mu)$, where $V$ is a finite vertex set.  It is a bounded positive operator.
	\item \textbf{The minimal Laplacian on an infinite graph.}  
	One begins with 
	\[
	-\Delta_{\min}\colon C_c(V)\subset \ell^2(V)\to\ell^2(V),
	\]
	defined on finitely supported functions, which admits a unique essentially self‐adjoint extension $-\Delta$ on $\ell^2(V,\mu)$.
\end{itemize}
 In what follows, we will uniformly denote the Laplacian operator in all cases by $-\Delta$.

By the spectral theorem for the positive operator \(-\Delta\) on the Hilbert space \(\ell^{2}(V,\mu)\), there exists a unique projection–valued measure
\[
E:\mathcal B\bigl([0,\infty)\bigr)\;\longrightarrow\;
\Bigl\{\text{orthogonal projections on }\ell^{2}(V,\mu)\Bigr\},
\]
supported on the spectrum \(\sigma(-\Delta)\subset[0,\infty)\), such that
\[
-\Delta
=
\int_{[0,\infty)}\lambda\,dE(\lambda).
\]
In particular, for any \(u,v\in\ell^{2}(V,\mu),\) the complex measure
\[
E_{u,v}(B)
:=\langle E(B)\,u,\;v\rangle_{\ell^{2}(V,\mu)},
\quad B\subset[0,\infty)\ \text{Borel},
\]
has total variation
\(\lvert E_{u,v}\rvert([0,\infty))\le\|u\|\,\|v\|\).

Given any Borel function \(\varphi:[0,\infty)\to\mathbb{R}\), one defines the operator
\[
\varphi(-\Delta)
:=\int_{[0,\infty)}\varphi(\lambda)\,dE(\lambda)
\]
with domain
\[
\operatorname{Dom}\bigl(\varphi(-\Delta)\bigr)
=\Bigl\{\,u\in\ell^{2}(V,\mu):\int_{[0,\infty)}\bigl|\varphi(\lambda)\bigr|^2
\,dE_{u,u}(\lambda)<\infty\Bigr\}.
\]
Its action is characterized by the identity
\[
\langle \varphi(-\Delta)u,\;v\rangle
=\int_{[0,\infty)}\varphi(\lambda)\,dE_{u,v}(\lambda),
\quad u\in\operatorname{Dom}(\varphi(-\Delta)),\;v\in\ell^{2}(V,\mu),
\]
and the associated norm identity
\[
\bigl\|\varphi(-\Delta)u\bigr\|^2
=\int_{[0,\infty)}\bigl|\varphi(\lambda)\bigr|^2\,dE_{u,u}(\lambda).
\]

If $\varphi$ is bounded, then $$\|\varphi(-\Delta)\|=\|\varphi\big|_{\sigma\left(-\Delta\right)}\|_{L^{\infty}};$$
and if $\varphi$ is real–valued (resp.\ non–negative), 
then $\varphi(-\Delta)$ is self–adjoint (resp.\ positive).

For $s\ge 0$ define
\[
(-\Delta)^{s}
\;=\;
\int_{[0,\infty)}\lambda^{s}\,\mathrm dE(\lambda),
\]
with 
\[\operatorname{Dom}\left(\left(-\Delta\right)^s\right)=\Bigl\{u\in\ell^{2}(V,\mu):
\int_{[0,\infty)}\lambda^{2s}\,\mathrm dE_{u,u}(\lambda)<\infty
\Bigr\}:=H^{2s}(V),\]
Endowed with the inner product
\[
\langle u,v\rangle_{H^{2s}}
\;=\;
\int_{[0,\infty)}(1+\lambda)^{2s}\,\mathrm dE_{u,v}(\lambda),
\qquad
\|u\|_{H^{2s}}^2
=\langle u,u\rangle_{H^{2s}},\] 
It is obvious that \(H^{s}(V)\) is a inner product  space and
\[
u\in H^{2s}(V)
\;\Longleftrightarrow\;
(-\Delta)^{s}u\in\ell^{2}(V,\mu), u\in \ell^{2}(V,\mu)
\]
since the homogeneous Sobolev seminorm
\[
\|(-\Delta)^{s}u\|_{\ell^{2}\left(V,\mu\right)}^{2}
=
\int_{[0,\infty)}\lambda^{2s}\,\mathrm dE_{u,u}(\lambda).
\]
Thus, $H^{s}\left(V\right)$ is continuously embedded in $\ell^{2}\left(V,\mu\right).$ In particular, if $-\Delta$ is a unique essential self-adjoint extension of $-\Delta_{min}$ on infinite graphs, then  $$C_c\left(V\right)\subset H^{2}\left(V\right)=\operatorname{Dom}\left(-\Delta\right)\subset H^{2s}\left(V\right),s\in \left(0,1\right).$$

Suppose the bottom of the spectrum of $-\Delta$ is \(\lambda_{0}>0\), then
\[
\sigma(-\Delta)\;\subset\;[\lambda_{0},\infty).
\]
By the functional calculus,
\[
\|(-\Delta)^{s}u\|_{\ell^{2}\left(V,\mu\right)}^{2}
=\int_{[0,\infty)}\lambda^{2s}\,dE_{u,u}(\lambda)
\;=\;
\int_{[\lambda_0,\infty)}\lambda^{2s}\,dE_{u,u}(\lambda)
\ge \lambda_{0}^{2s}\,\|u\|_{\ell^{2}\left(V,\mu\right)}^{2}.
\]
Hence
\[
\|(-\Delta)^{s}u\|_{\ell^{2}\left(V,\mu\right)}
\;\ge\;
\lambda_{0}^{\,s}\,\|u\|_{\ell^{2}\left(V,\mu\right)},
\]
for all $u\in H^{2s}\left(V\right)$. Hence, the homogeneous Sobolev seminorm and the \(H^{s}\)‐norm are equivalent on \(H^{s}\left(V\right)\).

The condition \(\inf\sigma(-\Delta)=\lambda_{0}>0\) is not met by finite graphs (for which the constant function belongs to \(\ell^{2}(V,\mu)\) and gives eigenvalue \(0\)).  It is, however, satisfied in several natural infinite settings, such as infinite \(k\)-regular trees, \(k\ge3\).  For the (unnormalised) Laplacian on the infinite \(k\)-regular tree \(\mathbb T_k\) one has 
	\[
	\sigma(-\Delta)=\bigl[k-2\sqrt{k-1},\,k+2\sqrt{k-1}\bigr],
	\]
	so the bottom of the spectrum is
	\(\lambda_{0}=k-2\sqrt{k-1}>0\).

	A positive spectral gap is, however, merely a sufficient condition for this equivalence.
	Many graphs with \(\inf\sigma(-\Delta)=0\), for instance, lattices \(\mathbb Z^{d}\)
	endowed with standard Laplacian, or graphs satisfying suitable volume-growth
	and Poincaré inequalities, still enjoy this  equivalence, see \cite[Lemma 2.1]{wang2024fractional}.

We now state the following convergence results, which can be proved via functional calculus; for detailed proofs, see \cite{chen2025logarithmic}.

\begin{proposition}\label{slx1}
	For every \(u\in H^{\epsilon}(V)\), \(\epsilon>0\), 
	\[
	\lim_{s\to0^{+}}\bigl\|(-\Delta)^{s}u - u+E(\{0\})u\bigr\|_{\ell^{2}(V,\mu)} = 0.
	\]
\end{proposition}

\begin{proposition}\label{slx2}
	For every \(u\in H^2(V)\), 
	\[
	\lim_{s\to1^{-}}\bigl\|(-\Delta)^{s}u + \Delta \: u\bigr\|_{\ell^{2}(V,\mu)} = 0.
	\]
\end{proposition}

\begin{proposition}\label{slx3}
	If \(u\in H^\varepsilon(V)\cap H^{\log}(V)\), then
	\[
	\frac{(-\Delta)^s - I + E(\{0\})}{s}\,u
	\xrightarrow{s\to0^+}
	\log(-\Delta)\,u
	\quad\text{in }\ell^2(V,\mu),
	\]
	where $H^{\log}(V)$ is defined in Proposition \ref{hlog}.
\end{proposition}

Since the vertex set \(V\) is discrete, convergence in \(\ell^2(V,\mu)\) implies convergence in \(\ell^\infty(V)\), and hence the above limits also hold pointwise at each \(x\in V\).

\section{Logarithmic Laplacian on Graphs}

In this section, we use functional calculus to derive explicit formulas for both the fractional Laplacian and the logarithmic Laplacian in the graph setting.  We begin by recalling the spectral definition of $(-\Delta)^s$ on finite graphs and infinite weighted graphs.  Next, we define the logarithmic Laplacian on finite graphs by functional calculus, and on infinite graphs we give a Bochner‐integral representation of the logarithmic Laplacian.  To our knowledge, this is the first rigorous construction of the logarithmic Laplacian on infinite graphs.

\subsection{Fractional Laplacian on Finite and Infinite Graphs}

Since on a finite connected graph \(G=(V,E,\mu,w)\), the Laplacian \(-\Delta\) is a finite dimensional self–adjoint operator, its spectrum consists of eigenvalues
\[
0=\lambda_{0}<\lambda_{1}\le\lambda_{2}\le\cdots\le\lambda_{N-1},
\]
with corresponding orthonormal eigenfunctions \(\{\varphi_{j}\}_{j=0}^{N-1}\) in \(\ell^2(V,\mu)\), where \(N = |V|\) is the number of vertices of \(G\). The spectral measure \(E(\lambda)\) is then a projection measure supported on \(\{\lambda_j\}\), and
\[
(-\Delta)^{s}
=\int_{[0,\infty)}\lambda^{s}\,dE(\lambda)
=\sum_{j=0}^{N-1}\lambda_j^{s}\,E(\{\lambda_j\}).
\]
Hence for any \(u\in\ell^2(V,\mu)\) and \(x\in V\),
\[
(-\Delta)^{s}u(x)
=\sum_{j=0}^{N-1}\lambda_j^{s}\,\langle u,\varphi_j\rangle\,\varphi_j(x).
\]

We now consider the more general setting of infinite weighted graphs $G=(V,E,w,\mu)$ for which the minimal Laplacian $\Delta_{\mathrm{min}}$ is essentially self‐adjoint.

Let 
\[
\varphi_t(\lambda)=e^{-t\lambda},\quad t>0,
\]
and define the heat semigroup by functional calculus as
\[
e^{t\Delta}:=\varphi_t(-\Delta),
\]
which is a strongly continuous contraction on \(\ell^2(V,\mu)\) with generator \(-\Delta\).  Equivalently, for each \(u\in\mathrm{Dom}(-\Delta)\), the function \(u_t=e^{t\Delta}u\) is the unique solution of
\[
\begin{cases}
	\partial_t u_t = \Delta u_t, & t>0,\\
	u_0 = u.
\end{cases}
\]

It is well known that the solution \(u_t = e^{t\Delta}u\) admits the following kernel representation:
\[
e^{t\Delta}u(x)
\;=\;
\sum_{y\in V}p(t,x,y)\,u(y)\,m(y),
\quad x\in V,\;t\ge0,
\]
where \(p(t,x,y)\) is the associated heat kernel. A straightforward computation shows that
\[
p(t,x,y)
\;=\;
\frac{1}{m(x)\,m(y)}\,
\left\langle \mathbf{1}_{\{x\}},\,e^{t\Delta}\mathbf{1}_{\{y\}} \right\rangle_{\ell^2(V,\mu)},
\]
and \(\mathbf{1}_{\{x\}}\) denotes the indicator function at the vertex \(x\).

For \(0<s<1\), one uses the scalar identity
\[
\lambda^s
=\frac{s}{\Gamma(1-s)}
\int_0^\infty \bigl(1 - e^{-t\lambda}\bigr)\,t^{-1-s}\,dt,
\quad \lambda\ge0,
\]
and applies functional calculus to \(-\Delta\), the following Proposition \ref{fenshubo} holds, detailed proofs can be found in \cite[Proposition 2.10]{chen2025logarithmic}:

	\begin{proposition}\label{fenshubo}
	Let $0<s<1$, for each $u\in H^{2s}(V)$ and $v\in \ell^2(V,\mu)$, define
	\[
	A
	:=\int_{\sigma(-\Delta)}\!\int_{0}^{\infty}(1-e^{-t\lambda})\,t^{-1-s}\,dt\;dE(\lambda)\,,
	\]
	\[
	B 
	:=\int_{0}^{\infty}\!\int_{\sigma(-\Delta)}(1-e^{-t\lambda})\,dE(\lambda)\;\,t^{-1-s}dt\;.
	\]
	Then for all $f,g$
	\[
	\bigl\langle A u,\,v\bigr\rangle_{\ell^2}
	=\bigl\langle B u,\,v\bigr\rangle_{\ell^2},
	\]
	and hence $A=B$ as operators on $H^{2s}(V)$ in the strong operator topology.
\end{proposition}

Hence, we obtain the well–known Bochner integral representation of the fractional Laplacian: for every \(u\in H^{2s}(V)\) and \(0<s<1\),  
\begin{equation}\label{bochfrac}
	(-\Delta)^{s}u=\frac{s}{\Gamma(1-s)}
	\int_{0}^{\infty}\bigl(u-e^{t\Delta}u\bigr)\,t^{-1-s}\,dt,u\in H^{2s}(V)
\end{equation}

M. Zhang \cite{zhang2024fractional1} established an pointwise description of the fractional Laplacian on stochastically complete locally finite weight  graphs by Bochner formula (\ref{bochfrac}), which extends domain $H^{2s}\left(V\right)$ to all \(u\in \ell^\infty(V)\).  Specifically, for any \( u \in \ell^\infty(V) \), the fractional Laplacian defined in \eqref{bochfrac} admits the pointwise representation
\begin{equation} \label{eq:kernel-form}
	(-\Delta)^s u(x) = \frac{1}{\mu(x)} \sum_{\substack{y \in V , y \neq x}} W_s(x, y)(u(x) - u(y)),
\end{equation}
where the fractional weight \( W_s(x, y) \) is given by
\begin{equation} \label{eq:weight}
	W_s(x, y) = \frac{s}{\Gamma(1 - s)}\mu(x)\mu(y) \int_0^{+\infty} p(t, x, y) t^{-1-s} \, dt, \quad \forall \:\:x\ne y \in V.
\end{equation}
 The weight function \( W_s(x, y) \) is symmetric and positive, i.e., \( W_s(x, y) = W_s(y, x) > 0 \).

\subsection{Logarithmic Laplacian on Finite and Infinite Graphs}
We begin by recalling some fundamentals of functional calculus. By the spectral theorem, for each \(s\in (0,1)\), the fractional Laplacian  of the positive  operator \(-\Delta\) is
\[
(-\Delta)^{s}u
=\int_{[0,\infty)}\lambda^{s}\,dE(\lambda)\,u,\quad u\in H^{2s}\left(V\right)
\]
and
\[
(-\Delta)^{s}u - u+E\left(\left\{0\right\}\right)u
=\int_{0}^{\infty}\bigl(\lambda^{s}-1\bigr)\,dE(\lambda)\,u.
\]
By Proposition \ref{slx1}
and note that
\[
\frac{\lambda^{s}-1}{s}
\;\xrightarrow{s\to0^+}\;\log \lambda,\:\lambda>0,
\]
it is natural to define the logarithmic Laplacian by
\begin{equation}\label{scja}
	\log(-\Delta)
	:=\int_{0}^{\infty}\log\lambda\;dE(\lambda).
\end{equation}
By spectral calculus,
\[
\operatorname{Dom}\bigl(\log(-\Delta)\bigr)=
\Bigl\{\,u\in\ell^{2}(V,\mu)\;\Big|\;
\int_{0}^{\infty}(\log\lambda)^{2}\,dE_{u,u}(\lambda)<\infty
\Bigr\}.
\]

\begin{proposition}\cite[Proposition 2.2]{chen2025logarithmic}\label{hlog}
	Equipped with the inner product
	\[
	\langle u,v\rangle_{\log}
	:=\langle u,v\rangle_{\ell^{2}}
	+\bigl\langle\log(-\Delta)u,\;\log(-\Delta)v\bigr\rangle_{\ell^{2}},
	\]
	the space
	\[
	\operatorname{Dom}\bigl(\log(-\Delta)\bigr)
	=\Bigl\{u\in\ell^{2}(V,\mu)\;\Big|\;\int_{0}^{\infty}(\log\lambda)^{2}\,dE_{u,u}(\lambda)<\infty\Bigr\}:=H^{\log}(V)
	\]
	is a Hilbert space in $\ell^2\left(V,\mu\right).$
\end{proposition}

The following Bochner integral  for the logarithmic Laplacian holds:

\begin{theorem}\cite[Theorem 2.12]{chen2025logarithmic}\label{bochlog}
	For every \(u\in H^{\log}(V)\),
	\[
	\log(-\Delta)\,u
	=\int_{0}^{\infty}\frac{e^{-t}u - e^{t\Delta}u}{t}\,dt,
	\]
	where the Bochner integral converges in \(\ell^2(V,\mu)\).
\end{theorem}

Below, we derive the expressions for the logarithmic Laplacian on finite and infinite graphs, respectively, using functional calculus and the Bochner integral formula.

Let \(G=(V,E,\mu,w)\) be a finite, connected weighted graph with \(|V|=N\), and let
\[
0=\lambda_{0}<\lambda_{1}\le\cdots\le\lambda_{N-1}
\]
be the eigenvalues of \(-\Delta\), with orthonormal eigenfunctions \(\{\varphi_{j}\}_{j=0}^{N-1}\) in \(\ell^2(V,\mu)\).  The projection onto the zero‐eigenspace is
\[
E(\{0\})u \;=\;\langle u,\varphi_{0}\rangle\,\varphi_{0},
\qquad \varphi_{0}(x)=\frac{1}{\sqrt{\sum_{z}\mu(z)}}.
\]

Since $-\Delta$ is self–adjoint with purely discrete spectrum  on finite graphs, and $E(\{\lambda_j\})$ denotes the corresponding orthogonal projections, it follows from \eqref{scja} and the spectral theorem that
\[
\log(-\Delta)\,u
=\int_{0}^{\infty}\log\lambda\,dE(\lambda)\,u
=\sum_{j=1}^{N-1}\log\lambda_{j}\,\langle u,\varphi_{j}\rangle\,\varphi_{j}.
\]

In what follows, let $G=(V,E,\mu,w)$ be an infinite, connected, stochastically complete weighted graph. We will then derive a pointwise formula for the logarithmic Laplacian on $G$. To arrive at the pointwise formula, we begin by stating a lemma that will be invoked in the proof of Theorem \ref{pointlog11}, which can be found in \cite{chen2025logarithmic}.

\begin{lemma}\label{euler}
	The following identity holds:
	\[
	\int_{0}^{1}\frac{e^{-t}-1}{t}\,dt
	+\int_{1}^{\infty}\frac{e^{-t}}{t}\,dt
	=-\gamma,
	\]
	where \(\gamma:=-\Gamma^{\prime}(1)\) is the Euler–Mascheroni constant.
\end{lemma}

Next, we give the proof of our first main result.
		
		\vspace{1\baselineskip}
\noindent \textbf{Proof of Theorem \ref{pointlog11}:}
	
	We use the Bochner formula in Theorem \ref{bochlog} and split at \(t=1\):
	\[
	\log(-\Delta)u
	=\int_{0}^{\infty}\frac{e^{-t}u - e^{t\Delta}u}{t}\,dt.
	\]
	\textbf{(1) Short times \(0<t<1\).}  Write
	\[
	e^{-t}u(x)-e^{t\Delta}u(x)
	=(e^{-t}-1)u(x)
	+\sum_{y\in V,y\ne x}\bigl(u(x)-u(y)\bigr)p(t,x,y)\,\mu(y),
	\]
	by using \(e^{t\Delta}u(x)=\sum_y p(t,x,y)u(y)\mu(y)\).  Hence
\begingroup\small		\[
	\int_0^1\frac{e^{-t}u(x)-e^{t\Delta}u(x)}{t}\,dt
	=\frac{1}{\mu(x)}\int_0^1 \sum_{y\neq x}
	\mu(x)\mu(y)\bigl(u(x)-u(y)\bigr) \tfrac{p(t,x,y)}{t}dt
	\;+\int_0^1\frac{e^{-t}-1}{t}\,dt u(x).
	\]\endgroup
	Note that
	\[\sum_{y\ne x}p\left(t,x,y\right)\mu(y)=1-p\left(t,x,x\right)\mu(x)\le \Bigl(\mu(x)\,\max_{t\in[0,1]}\bigl|\partial_{t}p(t,x,x)\bigr|\Bigr)\,t.\]
	Thus, since $u\in C_c\left(V\right),$ by the dominated convergence theorem,
	\[
	\int_0^1\frac{e^{-t}u(x)-e^{t\Delta}u(x)}{t}\,dt
	= \frac{1}{\mu(x)}\sum_{\substack{y\in V,y\neq x}}
	W_{\log}(x,y)\bigl(u(x)-u(y)\bigr)
	\;+\int_0^1\frac{e^{-t}-1}{t}\,dt u(x).
	\]
	
	\medskip\noindent
	\textbf{(2) Long times \(t\ge1\).}  
	Integrating over \([1,\infty)\) gives
	\[
	\int_1^\infty\frac{e^{-t}u-e^{t\Delta}u}{t}\,dt
	=\int_1^\infty\frac{e^{-t}}{t}\,dtu(x)-\frac{1}{\mu\left(x\right)}\int_1^\infty \sum_{y\in V}p(t,x,y)u(y)\mu(y)\mu(x)\frac{dt}{t} .
	\]
Since $u\in C_c\left(V\right),$ the sum is finite. Combining (1),\,(2) and Lemma \ref{euler}, we obtain
\begingroup\small	\[
\bigl(\log(-\Delta)u\bigr)(x)
=\frac{1}{\mu(x)}\sum_{\substack{y\in V,y\neq x}}
W_{\log}(x,y)\bigl(u(x)-u(y)\bigr)
\;-\;\frac{1}{\mu(x)}\sum_{y\in V}W(x,y)\,u(y)
\;+\;\Gamma'(1)\,u(x).
\] \endgroup
 \hfill $\square$

\section{Logarithmic Laplacian on Lattice Graphs}

In this section, we first recall key estimates for the heat kernel on weighted graphs and show how these bounds lead to precise control of the kernel functions for the logarithmic Laplacian.  We then derive a number of convergence and integrability results for both the fractional and logarithmic Laplacians.  Finally, we review the discrete Fourier transform on the integer lattice, identify the Fourier multipliers of the fractional and logarithmic operators, and use them to analyze asymptotics of the associated diffusion kernels.

\subsection{Heat Kernel Estimates on Graphs}\label{rehe}

Heat‐kernel bounds on Riemannian manifolds have a long history (see, e.g., \cite{davies1989heat,grigor2009heat,li1986parabolic,ondiag}).  On non-negative Ricci curvature manifolds, the Li–Yau estimate gives two‐sided Gaussian bounds
\[
\frac{C_l}{V(x,\sqrt t)}\exp\Bigl(-c_l\,\tfrac{d(x,y)^2}{t}\Bigr)
\;\le\;p(t,x,y)\;\le\;
\frac{C_r}{V(x,\sqrt t)}\exp\Bigl(-c_r\,\tfrac{d(x,y)^2}{t}\Bigr).
\]

	This line of work has spurred numerous advances in the graph‐theoretic setting.  For example, Bauer et al.\ \cite{bauer2015li} established a discrete Li–Yau inequality under the curvature condition $CDE(n,0)$ and derived corresponding heat‐kernel estimates.  Although their upper bounds exhibit Gaussian‐type decay, the lower bounds depend explicitly on the dimension parameter $n$ and fail to be genuinely Gaussian (see Theorem \ref{gau1}).  To overcome this, Horn et al.\ \cite{horn2019volume} introduced the strengthened curvature condition $CDE'(n,0)$ and proved two‐sided Gaussian bounds on graphs (see Theorem \ref{gaus1}).
	
	Subsequent work by Lin et al.\ \cite{lin2016global,lin2017gradient} obtained gradient estimates for positive solutions on graphs and used them to refine heat‐kernel bounds.  In parallel, Davies \cite{davies1993large} derived non‐Gaussian upper estimates in the continuous‐time graph setting via Legendre‐transform techniques.  Bauer et al.\ \cite{bauer2017sharp} later proved a sharp Davies–Gaffney–Grigor’yan lemma on graphs, recovering Davies’s bounds as an application, and Folz \cite{folz2011gaussian} used these tools to establish long‐range weak Gaussian estimates with respect to adapted metrics.
	
	Lower bounds, in particular, pose greater challenges.  A widely studied form is the on‐diagonal estimate
	\[
	p(t,x,x)\;\ge\;\frac{C_1}{V\bigl(x,C_2\sqrt t\bigr)},
	\]
	which has been treated in both manifold and graph contexts.  Delmotte \cite{delmotte1999parabolic} showed that this holds on graphs satisfying a continuous‐time parabolic Harnack inequality $\mathcal H(\eta,\theta_1,\dots,\theta_4,C)$.  Horn et al.\ \cite{horn2019volume} later recovered analogous lower bounds under $CDE'(n,0)$.  More recently, Lin et al.\ \cite{lin2017diagonal} proved that a purely volume‐growth hypothesis suffices to yield
	\[
	p(t,x,x)\;\ge\;\frac1{4\,V\!\bigl(x,\sqrt{Ct\log t}\bigr)}
	\quad\text{for all sufficiently large }t.
	\]
	
	Delmotte further established the equivalence
	\[
	\mathcal H(\eta,\theta_1,\dots,\theta_4,C_H)
	\;\Longleftrightarrow\;
	VD(C_1)+P(C_2)
	\;\Longleftrightarrow\;
	G(c_l,C_l,C_r,c_r)
	\]
	on graphs (with the manifold analogue due to Grigor’yan–Saloff‐Coste).  In the manifold case, nonnegative Ricci curvature implies both volume doubling and the Poincaré inequality.  In \cite{bauer2015li}, it was shown that $CDE(n,0)$ implies $\mathcal H$ (and hence all equivalent properties) provided $G$ admits a $(c,\eta R)$ strong cut‐off function in every ball of radius $R$ (for instance, $\mathbb Z^d$ admits such a function with parameter $1/\sqrt d$).
	
	Horn et al.\ \cite{horn2019volume} proved that $CDE'(n,0)$ alone suffices to guarantee volume doubling, a Poincaré inequality, Gaussian heat‐kernel bounds, and the continuous‐time Harnack inequality.  While the Harnack‐inequality approach in \cite{bauer2015li} yields a Gaussian upper bound for bounded‐degree graphs under $CDE(n,0)$, no matching lower bound arises because $CDE(n,0)$ does not by itself imply volume doubling (see Theorem \ref{gau1}).
	
	\begin{theorem}\label{gau1}
		Suppose $G$ satisfies $CDE(n,0)$ and has maximum degree $D$.  Then there are constants $C,C',C''>0$ such that for all $t>1$,
		\[
		C\,t^{-n}\exp\!\bigl(-C'\tfrac{d(x,y)^2}{t-1}\bigr)
		\;\le\;p(t,x,y)\mu(y)\;\le\;
		C''\,\frac{\mu(y)}{V\bigl(B(x,\sqrt t)\bigr)}.
		\]
	\end{theorem}
	
	Under the strengthened condition $CDE'(n_0,0)$ together with the loop‐and‐weight hypothesis $\Delta(\alpha)$, one recovers full two‐sided Gaussian estimates:
	
	\begin{theorem}\label{gaus1}
		If $G$ satisfies $CDE'(n_0,0)$ and $\Delta(\alpha)$, then there exist constants $C,C',c'>0$ (depending only on $n_0,\alpha$) such that for all $x,y\in V$ and $t>0$,
		\[
		\frac{C'}{V(x,\sqrt t)}
		\exp\!\bigl(-c'\tfrac{d(x,y)^2}{t}\bigr)
		\;\le\;
		p(t,x,y)
		\;\le\;
		\frac{C}{V(x,\sqrt t)}.
		\]
	\end{theorem}
	
	These Gaussian bounds are instrumental in our analysis of heat propagation, regularity properties, and kernel estimates for the logarithmic Laplacian on graphs.  
	
	We now turn to a finer analysis of the logarithmic Laplacian, which hinges on sharp heat‐kernel estimates.  To this end, we restrict our attention to weighted lattice graphs satisfying
	$\mu_{\min}>0, \mu_{\max}:=\sup_{x\in \mathbb{Z}^d}\mu(x)<\infty$, equipped with th enormalized Laplacian with $\mu(x)=m(x)$, then the loop‐and‐weight condition $\Delta(\alpha)$ defined by Delmotte in \cite{delmotte1999parabolic} is satisfied: for some $\alpha > 0$, the graph $G$ satisfies $\Delta(\alpha)$ if
\begin{enumerate}
	\item $x \sim x$ for every $x \in V$ (i.e., every vertex has a loop), 
	\item for any $x, y \in V$ with $x \sim y$, we have $w_{xy} \geq \alpha m(x)$.
\end{enumerate}
	
These assumptions in particular imply $w_{\min}>0$ and also ensure that the minimal Laplacian $\Delta_{\mathrm{min}}$ is essentially self‐adjoint.

	The condition $\Delta(\alpha)$  is very mild: whenever $\mu_{\max}<\infty$ and $w_{\min}:=\inf_{x \sim y}w_{xy}>0$, one automatically has
	$\Delta(\tfrac{w_{\min}}{\sup_x\mu(x)})$.  Moreover, adding loops does not degrade curvature bounds or alter key geometric invariants (e.g.\ volume growth or diameter), so any graph may be modified to satisfy $\Delta(\alpha)$ without affecting the subsequent analysis.

Under these hypotheses and in light of Theorem \ref{gaus1}, the continuous‐time heat kernel \(p(t,x,y)\) on \(\mathbb{Z}^d\) satisfies estimates (\ref{gaussian})(\ref{222})(\ref{444})(\ref{333}).

Since \(\frac{m(x)}{\mu(x)}=1\), the normalized Laplacian \(\Delta\) is bounded on \(\ell^2(\mathbb{Z}^d,\mu)\) (see \cite[Theorem~1.27]{keller2021graphs}). Consequently, the heat semigroup preserves constants,
\[
e^{t\Delta}\mathbf{1} = \mathbf{1}
\quad\forall\,t\ge0,
\]
which is equivalent to stochastic completeness:
\[
\sum_{y\in\mathbb{Z}^d}p(t,x,y)\,\mu(y)=1,
\quad\forall\,t\ge0,\;x\in\mathbb{Z}^d.
\]

Under our assumptions, \(\mu_{\min}>0\) and
$\mu_{\max}<\infty$ yield the polynomial volume growth
\[
V(x,r)\;\sim\;r^d,
\quad r>0.
\]

\subsection{Kernel Estimates for the Logarithmic Laplacian}

With the precise heat‐kernel bounds in hand, we can now derive sharp estimates for the logarithmic kernel.  Recall that on \(\mathbb{Z}^d=\left(V,E,\mu,w\right)\),  we set
\[
W_{\log}(x,y)
=\mu(x)\mu(y)\int_0^1 p(t,x,y)\,\frac{dt}{t},
\]
and
\[
W(x,y)
=\mu(x)\mu(y)\int_1^\infty p(t,x,y)\,\frac{dt}{t}.
\]

Now, we give the proof of Proposition \ref{prop:Wlog11} and Proposition \ref{prop:Wxy11}. 

	\vspace{1\baselineskip}
	
\noindent \textbf{Proof of Proposition \ref{prop:Wlog11}:}

	The identity in (\ref{wlog}) is immediate from
	\[
	\sum_{y\ne x}p(t,x,y)\mu(y)
	=1 - p(t,x,x)\mu(x),
	\]
	and exchanging sum and integral. 
	
	 To prove (\ref{wlogg}), set \(r=d(x,y)\ge1\).  Using estimate \eqref{444} and \(V(x,\sqrt t)\ge\mu_{\min}\), 
	\[
	p(t,x,y)\lesssim e^r\Bigl(\tfrac{t}{r}\Bigr)^r=\frac{e^r}{r^r}t^r.
	\]
	Hence
	\[
	W_{\log}(x,y)
	\lesssim \mu(x)\mu(y)\frac{e^r}{r^r}\int_0^1 t^{r-1}\,dt
	=\mu(x)\mu(y)\,\frac{e^r}{r^{r+1}},
	\]
	which gives the claimed bound.\hfill $\square$

	\vspace{1\baselineskip}

\noindent \textbf{Proof of Proposition \ref{prop:Wxy11}:}

	By the Gaussian bounds \eqref{gaussian} and $0<\mu_{\min}\le\mu_{\max}<\infty$, one obtains immediately	
	$$W(x,x)\sim 1.$$
	Next, for \(r=d(x,y)\ge1\), by the Gaussian bounds \eqref{gaussian}, there exists $c_1>0$ such that
	\[
	W(x,y)
	\;=\;\mu(x)\mu(y)\int_1^\infty p(t,x,y)\,\frac{dt}{t} \gtrsim \int_1^\infty t^{-\frac d2-1}
	e^{-c_1\frac{r^2}{t}}\,dt.
	\]
Set \(u=c_1r^2/t\), then we obtain that
		\[
	W(x,y) \gtrsim r^{-d}\int_0^{c_{1}r^2} t^{\frac d2-1}
	e^{t}\,dt\gtrsim r^{-d}.
	\]

	For the upper bound, split
	\[
	\int_1^\infty=\int_1^{\frac{r}{e}}+\int_{\frac{r}{e}}^\infty.
	\]
On the interval $[1,\frac{r}{e}]$, the bound provided by \eqref{444} is too coarse to produce the desired decay in $W(x,y)$.  Instead, we invoke the Sharp Davies–Gaffney–Grigor’yan Lemma \cite{bauer2017sharp}, which yields for all $t>0$ and $r=d(x,y)$ the refined estimate
$$
p(t,x,y)\;\le\;\exp\!\bigl(-\zeta_1(t,r)\bigr)\le \exp\left(r-r\log \frac{2r}{t}\right)=e^{r}\left(\frac{t}{2r}\right)^r
$$
with
$$
\zeta_1(t,r)
= r \mathrm{arsinh}\!\bigl(\tfrac{r}{t}\bigr)
- \sqrt{r^2 + t^2} + t.
$$
Then 
\[
\int_1^{\frac{r}{e}} p(t,x,y)\,\frac{dt}{t}\lesssim \left(\frac{e}{2r}\right)^{r}\int_1^{\frac{r}{e}} t^{r-1}dt=\left(\frac{e}{2r}\right)^{r}\frac{r^re^{-r}-1}{r}\lesssim r^{-d}.
\]

	On \([\frac{r}{e},\infty)\), estimate \eqref{222} yields
	\[
	\int_{\frac{r}{e}}^\infty p(t,x,y)\,\frac{dt}{t}\lesssim \int_{\frac{r}{e}}^\infty t^{-1-\frac d2}e^{-c\frac{r^2}{t}}\,dt=r^{-d}\int_0^{cre}t^{\frac{d}{2}-1}e^{-t}dt\lesssim r^{-d}.
	\]
	Combining these gives \(W(x,y)\lesssim r^{-d}\).\hfill $\square$

\begin{remark}
	Proposition \ref{prop:Wxy11} implies 
	\(\sum_{y\in V}W(x,y)=+\infty\).
\end{remark}

For \(x\in\mathbb{Z}^d\) with neighbors \(\{y_i\}_{i=1}^\infty\), we define the log-gradient as
\[
\nabla^{\log}u(x)
=:\Bigl(
\sqrt{\tfrac{W_{\log}(x,y_i)}{2\,\mu(x)}}\bigl(u(x)-u(y_i)\bigr)
\Bigr)_{i\ge1}.
\]

Next, we prove that $	\nabla^{\log}$ is bounded on \(\ell^2(\mathbb{Z}^d)\) and long‐range component of the logarithmic Laplacian induces an unbounded quadratic form on \(\ell^2(\mathbb{Z}^d)\).

Subsequently, we give the proof of Proposition \ref{prop:unbounded22} and Proposition \ref{prop:unbounded11}. 

	\vspace{1\baselineskip}
	
\noindent \textbf{Proof of Proposition \ref{prop:unbounded22}:}

	By definition,
	\[
	\|\nabla^{\log}u\|_{\ell^2}^2
	=\sum_{x\in\mathbb{Z}^d}\sum_{y\ne x}\frac{W_{\log}(x,y)}{2\,\mu(x)}\bigl|u(x)-u(y)\bigr|^2.
	\]
	Using \(\lvert u(x)-u(y)\rvert^2\le2\lvert u(x)\rvert^2+2\lvert u(y)\rvert^2\) and symmetry of \(W_{\log}\), one finds
	\[
	\|\nabla^{\log}u\|_{\ell^2}^2
	\le\sum_{x\in\mathbb{Z}^d}\lvert u(x)\rvert^2\sum_{y\ne x}\frac{W_{\log}(x,y)}{\mu(x)}.
	\]
	By the kernel bound and \(\mu(y)\le \mu_{\max}\),
	\[
	\sum_{y\ne x}\frac{W_{\log}(x,y)}{\mu(x)}\lesssim\sum_{y\ne x}d(x,y)^{-1}\Bigl(\tfrac{e}{d(x,y)}\Bigr)^{d(x,y)}.
	\]
	Grouping by \(r=d(x,y)\ge1\) and using the fact that \(\#\{y:d(x,y)=r\}\lesssim r^{d-1}\), we obtain
	\[
	\sum_{y\ne x}d(x,y)^{-1}\Bigl(\tfrac{e}{d(x,y)}\Bigr)^{d(x,y)}\lesssim \sum_{r=1}^\infty r^{d-2}\Bigl(\tfrac{e}{r}\Bigr)^r
	\;<\;\infty.
	\]
	Thus, there is a uniform bound \(C>0\), so that
	\(\sum_{y\ne x}W_{\log}(x,y)/\mu(x)\le C\).  It follows that
	\[
	\|\nabla^{\log}u\|_{\ell^2}^2
	\le C\sum_{x\in\mathbb{Z}^d}\lvert u(x)\rvert^2
	=C\,\|u\|_{\ell^2}^2.
	\]
	Thus, we obtain the desired result.\hfill $\square$

	\vspace{1\baselineskip}

\noindent \textbf{Proof of Proposition \ref{prop:unbounded11}:}
	Let 
	\[
	u_n(x)
	=\frac{\mathbf{1}_{\{d(x,0)\le n\}}(x)}{\sqrt{V(0,n)}},
	\]
	where \(V(0,n)=\sum_{d(x,0)\le n}\mu(x)\).  Then \(\|u_n\|_{\ell^2}=1\).  By Proposition \ref{prop:Wxy11},
	\[
	W(x,y)\gtrsim d(x,y)^{-d},\:x\ne y,
	\]
	so
	\[
	\sum_{x,y}u_n(x)u_n(y)W(x,y)
	\;\gtrsim\;\frac{1}{V(0,n)}\sum_{d(x,0)\le n,\:d(y,0)\le n}d(x,y)^{-d}.
	\]
	Fixing \(x\) with \(d(x,0)\le n\), the inner sum over \(y\ne x\) satisfies
	\[
	\sum_{\substack{y:\,d(y,0)\le n\\y\ne x}}d(x,y)^{-d}
	\;\ge\;\sum_{k=1}^{n}\frac{|S(k)|}{k^{d}}
	\;\gtrsim\;\sum_{k=1}^{n}\frac{k^{d-1}}{k^d}
	=\sum_{k=1}^n\frac1k,
	\]
	where
	$$
	S(k)=:\{\,y\in\mathbb{Z}^d : d(y,0)=k\}.
	$$
Hence,
	\[
	\sum_{x,y}u_n(x)u_n(y)W(x,y)
	\;\gtrsim \sum_{k=1}^n\frac1k
	\;\longrightarrow\;\infty,
	\]
	as claimed.\hfill $\square$

\subsection{Convergence and Integrability}

In Propositions \ref{slx1}–\ref{slx3} above, we have established convergence statements for the fractional and logarithmic Laplacians in the $\ell^2$‐norm, but at the cost of rather strong regularity assumptions on the underlying functions.  In the sequel, we will develop a completely different approach based on the pointwise integral representation of $(-\Delta)^s$, which allows us to weaken these function‐space hypotheses substantially.  Along the way, we introduce several auxiliary lemmas of independent interest-beyond their role in our convergence proofs, and conclude the section by analyzing the integrability properties of the logarithmic Laplacian kernel.

\begin{lemma}\label{wxy1}
	For $x \neq y \in \mathbb{Z}^d $, we have
	\[
	\mu(x) \mu(y) \, p'(0, x, y) = w_{xy}.
	\]
\end{lemma}

\begin{proof}
	Let $u_0 \in \ell^{\infty}(\mathbb{Z}^d)$ and define the heat semigroup by
	\[
	(e^{t\Delta} u_0)(x) := \sum_{z \in V} p(t, x, z) \, u_0(z) \, \mu(z).
	\]
	Then, by the definition of the generator $\Delta$, we have
	\[
	\left. \frac{d}{dt} \right|_{t=0} (e^{t\Delta} u_0)(x) = \Delta\: u_0(x).
	\]
	Now, choose $u_0 = \delta_y$, the Dirac function at $y$. Then
	\[
	(e^{t\Delta} \delta_y)(x) = p(t, x, y) \mu(y),
	\]
	which gives
	\[
	\left. \frac{d}{dt} \right|_{t=0} (e^{t\Delta} \delta_y)(x) = p'(0, x, y) \mu(y).
	\]
	On the other hand, using the definition of $\Delta$ on Dirac functions, we have
	\[
	\Delta \delta_y(x) = \frac{w_{xy}}{\mu(x)}.
	\]
	Combining both expressions for the derivative at $t = 0$, we obtain
	\[
	p'(0, x, y) \mu(y) = \frac{w_{xy}}{\mu(x)},
	\]
	which yields the desired identity:
	\[
	\mu(x) \mu(y) \, p'(0, x, y) = w_{xy}.
	\]
\end{proof}

\begin{remark}
	The above result also applies to finite graphs.
\end{remark}

\begin{lemma}\label{higher‐derivative-nearest}
	For every integer \(k\ge 0\) and \(x,y\in\mathbb{Z}^d\),
	\[
	\sup_{t\in[0,1]}\,
	\bigl|\partial_t^k p(t,x,y)\bigr|
	\;\le\; \frac{2^k}{\mu_{min}}.
	\]
\end{lemma}

\begin{proof}
	We prove it by induction on \(k\in\mathbb N\).  
	For \(k=0\), the bound follows from \(0 \le p(t,x,y) \le \mu_{\min}^{-1}\).
	
	Assume the result holds for order \(k-1\). Using
	\[
	\partial_t^k p(t,x,y)
	= \Delta_x\bigl(\partial_t^{k-1}p(t,\cdot,y)\bigr)(x),
	\]
	we have
	\[
	\bigl|\partial_t^k p(t,x,y)\bigr|
	\le \frac{1}{\mu(x)}\sum_{z\sim x} w_{xz}
	\left| \partial_t^{k-1}p(t,z,y) - \partial_t^{k-1}p(t,x,y) \right|.
	\]
	By the induction hypothesis, \(|\partial_t^{k-1}p(t,x,y)| \le \frac{2^{k-1}}{\mu_{min}}\) uniformly on \([0,1]\).  
	Hence
	\[
	\sup_{t\in[0,1]}\,
	\bigl|\partial_t^k p(t,x,y)\bigr|
	\;\le\; \frac{2^k}{\mu_{min}},
	\]
	which is the desired inequality.
\end{proof}

Next estimate shows that, for \(d(x,y)\) large compared to the order of differentiation \(k\), the derivatives of the heat kernel exhibit rapid  decay in the distance \(d(x,y)\).

\begin{lemma}\label{rhyjx}
	For each integer \(k\ge0\), there exists a constant \(C_k>0\), such that for all distinct \(x,y\in\mathbb{Z}^d\) with $d(x,y)\ge k$, $\delta\in (0,1),$
	\[
	\sup_{t\in[0,\delta]} \bigl|\partial_t^k p(t,x,y)\bigr|
	\;\le\;\frac{C_k}{\bigl(d(x,y)-k\bigr)!}\,.
	\]
\end{lemma}

\begin{proof}
	We proceed by induction on \(k\).  Recall that the heat kernel satisfies the heat equation
	\[
	\partial_t p(t,x,y) \;=\;\Delta_x\,p(t,\cdot,y)(x),
	\]
	where \(\Delta\) is the normalized Laplacian
	\(\Delta f(x)=\tfrac1{\mu(x)}\sum_{z\sim x}w_{xz}\bigl(f(z)-f(x)\bigr)\).
	
	\medskip\noindent\textbf{Base case \(k=0\).}  
	The off–diagonal Gaussian‐type bound (\ref{333}) gives, for all \(t\in[0,\delta]\) 
	\[
	p(t,x,y)\le\;
	\frac{c}{d(x,y)!},
	\]
	where $c_1,c_2>0$,  depending only on the dimension \(d\) and $\mu_{min}$.
	
	\medskip\noindent\textbf{Inductive step.}  
	Assume the lemma holds for \(k-1\), i.e.\ there is \(C_{k-1}\) so that whenever \(d(x,y)\ge k-1\),
	\[
	\sup_{t\in[0,\delta]}\bigl|\partial_t^{\,k-1}p(t,x,y)\bigr|
	\;\le\;
	\frac{C_{k-1}}{\bigl(d(x,y)- (k-1)\bigr)!}.
	\]
	Then
	\[
	\partial_t^k p(t,x,y)
	=\partial_t\bigl(\partial_t^{\,k-1}p(t,x,y)\bigr)
	=\Delta_x\bigl(\partial_t^{\,k-1}p(t,\cdot,y)\bigr)(x).
	\]
	Hence
	\[
	\bigl|\partial_t^k p(t,x,y)\bigr|
	\;\le\;
	\frac1{\mu(x)}\sum_{z\sim x}w_{xz}
	\Bigl|\partial_t^{\,k-1}p(t,z,y)-\partial_t^{\,k-1}p(t,x,y)\Bigr|.
	\]
	For each neighbor \(z\sim x\), the graph‐distance to \(y\) satisfies
	\(
	d(z,y)\ge d(x,y)-1\ge k-1,
	\)
	so by the inductive hypothesis both
	\(\bigl|\partial_t^{k-1}p(t,z,y)\bigr|\) and
	\(\bigl|\partial_t^{k-1}p(t,x,y)\bigr|\) are bounded by
	\(C_{k-1}/\bigl(d(x,y)-k\bigr)!\).  Hence
	\[
	\bigl|\partial_t^k p(t,x,y)\bigr|
	\;\le\;
	\frac1{\mu(x)}\sum_{z\sim x}w_{xz}
	\;\,\frac{2C_{k-1}}{(d(x,y)-k)!}
	\;=\;
	\,\frac{2C_{k-1}}{(d(x,y)-k)!}.
	\]
	Setting \(C_k=2\,C_{k-1}\) completes the induction.
\end{proof}

\begin{lemma} \label{lem:uniform‐bound‐factorial} 
	Let \(k \in \mathbb{N}\). There exists a constant \(C>0\), such that
	\[
	\sup_{x\in \mathbb{Z}^d}\sum_{y:\,d(x,y)\ge k}\,\frac{1}{\bigl(d(x,y)-k\bigr)!}
	\;\le\; C.
	\]
\end{lemma}

\begin{proof}
	For each integer \(r \ge k\), the number of vertices at distance \(r\) from \(x\) satisfies
	\[
	\#\{y : d(x,y) = r\} \le C\,r^{d-1},
	\]
	for some constant \(C = C(d)\). Then
	\[
	\sum_{y:\,d(x,y)\ge k} \frac{1}{(d(x,y) - k)!}
	= \sum_{r=k}^{\infty} \frac{\#\{y : d(x,y) = r\}}{(r - k)!}
	\le A \sum_{r=k}^{\infty} \frac{r^{d-1}}{(r - k)!},
	\]
	which is finite and independent of \(x\), as desired.
\end{proof}

\begin{corollary}\label{yzslj}
	Let \(x \in \mathbb{Z}^d\),$\delta\in (0,1)$. Then the series
	\[
	\sum_{y \in \mathbb{Z}^d} \partial_t p(t, x, y) \mu(y)
	\]
	converges uniformly with respect to \(t \in [0,\delta]\), and for all \(t \in [0,\delta]\) we have
	\[
	\sum_{y \in \mathbb{Z}^d} \partial_t p(t, x, y) \mu(y) = 0.
	\]
	
\end{corollary}

\begin{proof}
	By Lemma \ref{rhyjx}, for $x,y\in \mathbb{Z}^d$ with $d\left(x,y\right)\ge 1$
	\[
	\left| \partial_t p(t,x,y) \right| \leq \frac{C}{\left(d(x,y)-1\right)!}, \:t\in [0,\delta].
	\]
	
	By Lemma \ref{lem:uniform‐bound‐factorial}, the series \(\sum_{y} \mathbf{1}_{\left\{d\left(x,y\right)\ge 1\right\}} \frac{1}{\left(d(x,y)-1\right)!}\) converges absolutely, the Weierstrass M-test implies that
	\[
	\sum_{y\in \mathbb{Z}^d} \partial_t p(t,x,y) \mu(y)
	\]
	converges uniformly in \(t\in[0,\delta]\).
	
	Hence, differentiation and summation can be interchanged:
	\[
	\partial_t \left( \sum_{y\in\mathbb{Z}^d} p(t,x,y)\mu(y) \right) = \sum_{y\in\mathbb{Z}^d} \partial_t p(t,x,y)\mu(y),
	\]
thus, by stochastically complete
	\[
	\sum_{y\in\mathbb{Z}^d} \partial_t p(t,x,y) \mu(y) = 0,
	\]
	as claimed.
\end{proof}

\begin{lemma}\label{stupbound}
	Fix \(x\in\mathbb{Z}^d\) and \(s\in(0,1)\).  Define
	\[
	S(t)
	=\sum_{\substack{y\in\mathbb{Z}^d,y\neq x}}
	\mu(y)\,p(t,x,y)\,t^{-1-s}.
	\]
	Set
	\[
	C_x \;=\;\mu(x)\,\max_{0\le\tau\le1}\bigl|\partial_\tau p(\tau,x,x)\bigr|.
	\]
	Then for all \(t>0\),
	\[
	S(t)\;\le\;
	\begin{cases}
		C_x\,t^{-s}, &0<t\le1,\\
		t^{-1-s},    &t>1.
	\end{cases}
	\]
\end{lemma}

\begin{proof}
	Since \(\sum_{y\in\mathbb{Z}^d }p(t,x,y)\mu(y)=1\), we have
	\[
	S(t)
	=\bigl(1-\mu(x)p(t,x,x)\bigr)\,t^{-1-s}.
	\]
	
	\emph{Case 1: \(0<t\le1\).}  By Taylor’s theorem,
	\[
	p(t,x,x)
	=p(0,x,x)+t\,\partial_\tau p(\tau,x,x)
	=\tfrac1{\mu(x)}+t\,\partial_\tau p(\tau,x,x)
	\quad\tau\in(0,1),
	\]
	so
	\[
	1-\mu(x)p(t,x,x)
	=-\,t\,\mu(x)\,\partial_\tau p(\tau,x,x)
	\;\le\;C_x\,t.
	\]
	Hence
	\[
	S(t)
	\le C_x\,t\cdot t^{-1-s}
	=C_x\,t^{-s}.
	\]
	
	\emph{Case 2: \(t>1\).}  Positivity of \(p\) gives \(0\le1-\mu(x)p(t,x,x)\le1\).  Thus
	\[
	S(t)
	\le t^{-1-s}.
	\]
	Combining both cases yields the claimed bounds.
\end{proof}

Thereafter, we proceed to prove two convergence results.

	\vspace{1\baselineskip}
	
\noindent \textbf{Proof of Proposition \ref{slx22211}:}

	Fix \(x\in \mathbb{Z}^d\). Enumerate \(\mathbb{Z}^d \setminus \{x\}\) as \(\{y_k\}_{k\geq 1}\), then all sums over \(y\neq x\) are written as sums over \(k\geq 1\).
	
	By definition,
	\[
	(-\Delta)^s u(x) = \frac{1}{\mu(x)} \sum_{k=1}^{\infty} W_s(x,y_k) (u(x) - u(y_k)),
	\]
	where
	\[
	W_s(x,y_k) = \frac{s}{\Gamma(1-s)} \mu(x)\mu(y_k) \int_0^{+\infty} p(t,x,y_k)t^{-1-s}\,dt.
	\]
	Thus,
	\[
	(-\Delta)^s u(x) + \Delta\: u(x) = \frac{1}{\mu(x)} \sum_{k=1}^{\infty} (W_s(x,y_k) - w_{xy_k})(u(x) - u(y_k)).
	\]
	Since \(u \in \ell^\infty(\mathbb{Z}^d)\), there exists \(M>0\) such that \(|u(x) - u(y_k)| \leq 2M\) for all \(k\).  
	Therefore,
	\[
	\left| (-\Delta)^s u(x) + \Delta u(x) \right| \leq \frac{2M}{\mu(x)} \sum_{k=1}^{\infty} \left|W_s(x,y_k) - w_{xy_k}\right|.
	\]
	
	Fix $\delta\in (0,1),$ we analyze \(\sum_{k=1}^{\infty} \left| W_s(x,y_k) - w_{xy_k} \right|\) by splitting \(W_s(x,y_k)\) into two parts:
	\[
	W_s(x,y_k) = W_s^{(1)}(x,y_k) + W_s^{(2)}(x,y_k),
	\]
	where
	\[
	W_s^{(1)}(x,y_k) = \frac{s}{\Gamma(1-s)} \mu(x)\mu(y_k) \int_0^\delta p(t,x,y_k)t^{-1-s}\,dt,\]
	and
	\[	W_s^{(2)}(x,y_k) = \frac{s}{\Gamma(1-s)} \mu(x)\mu(y_k) \int_\delta^{+\infty} p(t,x,y_k)t^{-1-s}\,dt.
	\]
	
	\noindent
	\textbf{Step 1: Estimate for large time integral $W_s^{(2)}$}
	
	By stochastic completeness, for all \(t\geq 0\),
	\[
	\sum_{k=1}^{\infty} p(t,x,y_k) \mu(y_k) \leq 1.
	\]
	It follows that
	\[
	\sum_{k=1}^{\infty} |W_s^{(2)}(x,y_k)| \leq \frac{s}{\Gamma(1-s)} \mu(x) \int_\delta^{+\infty} t^{-1-s}\,dt=\frac{\mu(x)\delta^{-s}}{\Gamma(1-s)}.
	\]

	\medskip
	
	\noindent\textbf{Step 2: Small‐time estimate via Taylor expansion with remainder.}
	
	For \(0\le t\le \delta\), the Taylor theorem gives
	\[
	p(t,x,y_k)
	= p(0,x,y_k) + t\,p'(0,x,y_k) + \tfrac12\,t^2\,p''(\theta_t,x,y_k).
	\]
	Since \(p(0,x,y_k)=0\) for \(k\ge1\), this reduces to
	\[
	p(t,x,y_k)
	= t\,p'(0,x,y_k) + \tfrac12\,t^2\,p''(\theta_t,x,y_k),\:\theta_t\in (0,\delta).
	\]
	Hence
	\[
	\int_0^\delta p(t,x,y_k)\,t^{-1-s}\,dt
	= p'(0,x,y_k)\int_0^\delta t^{-s}\,dt
	+ \tfrac12\int_0^\delta t^{1-s}\,p''(\theta_t,x,y_k)\,dt.
	\]
	By Lemma \ref{wxy1}, we obtain that
	\[
	\begin{aligned}
		W_s^{(1)}(x,y_k)-w_{xy_k}&
		=\mu(x)\mu(y_k)\,p'(0,x,y_k)\Bigl(\tfrac{s\delta^{1-s}}{(1-s)\Gamma(1-s)}-1\Bigr)
		\\&+\;\frac{s}{2\,\Gamma(1-s)}\,\mu(x)\mu(y_k)
		\int_0^\delta t^{1-s}\,p''(\theta_t,x,y_k)\,dt.
	\end{aligned}
	\]
	By Lemma \ref{wxy1}, $p'(0,x,y_k)\ge 0.$ By Lemma \ref{higher‐derivative-nearest} and Lemma \ref{rhyjx} , there exists $c_1,c_2>0$, such that
	\[\begin{aligned}
		\sum_{k=1}^\infty\bigl|W_s^{(1)}(x,y_k)-w_{xy_k}\bigr|\le &\sum_{k=1}^\infty \mu(x)\mu(y_k)\,p'(0,x,y_k)\left|\tfrac{s\delta^{1-s}}{(1-s)\Gamma(1-s)}-1\right|+\\&\frac{s\delta^{2-s}\mu(x)}{2\left(2-s\right)\,\Gamma(1-s)}\left(c_1+\sum_{k=1}^\infty \mathbf{1}_{\left\{d\left(x,y_k\right)\ge 2\right\}}
		\frac{c_2}{\left(d(x,y_k)-2\right)!}\right).
	\end{aligned}\]
	By Corollary \ref{yzslj}, we obtain that
	\[\bigl|\sum_{k=1}^\infty \mu(x)\mu(y_k)\,p'(0,x,y_{k})\bigr|=\mu(x)^{2}p^{\prime}\left(0,x,x\right).\]
	Hence,
	\[\begin{aligned}
		\left| (-\Delta)^s u(x) + \Delta u(x) \right| \leq&  \frac{2M\delta^{-s}}{\Gamma(1-s)}+2M\mu(x)p^{\prime}(0,x,x)\left|\tfrac{s\delta^{-s}}{(1-s)\Gamma(1-s)}-1\right|\\+&\frac{sM\delta^{2-s}}{\left(2-s\right)\,\Gamma(1-s)}\left(c_1+\sum_{k=1}^\infty \mathbf{1}_{\left\{d\left(x,y_k\right)\ge 2\right\}}
		\frac{c_2}{\left(d(x,y_k)-2\right)!}\right)
	\end{aligned}
	\]
	By Lemma \ref{lem:uniform‐bound‐factorial}, the last sum is uniformly bounded in \(x\).  Hence by Lemma \ref{higher‐derivative-nearest}
	\[
	\lim_{s\to1^-}\bigl\|(-\Delta)^s u + \Delta u\bigr\|_\infty = 0.
	\]
\hfill $\square$

\begin{remark}
	 Proposition~\ref{slx22211} is strictly stronger than Proposition~\ref{slx2}, as a result of \(H^2(\mathbb{Z}^d)\subset \ell^\infty(\mathbb{Z}^d)\).
\end{remark}

\noindent \textbf{Proof of Proposition \ref{slx11111}:}

	Fix \(x\in\mathbb{Z}^d\).  By definition,
	\[
	(-\Delta)^s u(x)-u(x)
	=\frac{1}{\mu(x)}\sum_{y\ne x}W_s(x,y)\bigl(u(x)-u(y)\bigr)\;-\;u(x).
	\]
	Rearrange as
	\[
	(-\Delta)^s u(x)-u(x)
	=\frac{1}{\mu(x)}\sum_{y\ne x}W_s(x,y)\,u(x)
	-\,u(x)
	\;-\;
	\frac{1}{\mu(x)}\sum_{y\ne x}W_s(x,y)\,u(y).
	\]
	Hence
	\[
	\bigl|(-\Delta)^s u(x)-u(x)\bigr|
	\;\le\;
	\left|\frac{1}{\mu(x)}\sum_{y\ne x}W_s(x,y)\,
	-1\right|u(x)
	\;+\;
	\frac{1}{\mu(x)}
	\sum_{y\ne x}W_s(x,y)\,\bigl|u(y)\bigr|.
	\]
	We split
	\[
	W_s(x,y)
	= W_s^{(1)}(x,y) + W_s^{(2)}(x,y),
	\]
	with
	\[
	W_s^{(1)}(x,y)=\frac{s}{\Gamma(1-s)}\,\mu(x)\mu(y)\int_0^1 p(t,x,y)\,t^{-1-s}dt\]
	and
	\[W_s^{(2)}(x,y)=\frac{s}{\Gamma(1-s)}\,\mu(x)\mu(y)\int_1^\infty p(t,x,y)\,t^{-1-s}dt.
	\]
	By Lemma \ref{stupbound} and stochastic completeness, so
	\[
	\sum_{y\in\mathbb{Z}^d,y\ne x}W_s^{(1)}(x,y)
	=\frac{s}{\Gamma(1-s)}\,\mu(x)\int_0^1\left(1-\mu(x)p(t,x,x)\right)t^{-1-s}\,dt.
	\]
	By the expansion
	\begin{equation}\label{taly1}
		p(t,x,x) = \frac{1}{\mu(x)} + t\,p'(\tau,x,x), \quad \tau \in (0,1),
	\end{equation}
	we obtain that
	\begin{equation}\label{ws11}
		\sum_{y\in\mathbb{Z}^d,y\ne x}W_s^{(1)}(x,y)
		=-\frac{s}{\Gamma(1-s)}\,\mu(x)^2\int_0^1p^{\prime}\left(\tau,x,x\right)t^{-s}\,dt.
	\end{equation}
	By Lemma \ref{higher‐derivative-nearest} and $u\in C_c\left(\mathbb{Z}^d\right)$, we obtain that	
	\[\lim\limits_{s\rightarrow 0^+}\sup_{x\in \mathbb{Z}^d}\frac{1}{\mu(x)}
	\sum_{y\ne x} W_s^{(1)}(x,y)\,\bigl|u(y)\bigr|=0.\]
	By the bound \eqref{gaussian} and the dominated convergence theorem, together with the fact that \(u\in C_c(\mathbb{Z}^d)\) is supported on only finitely many vertices, we obtain
	\[
	\lim_{s\to 0^+}\frac{1}{\mu(x)}
	\sum_{y\neq x} W_s^{(2)}(x,y)\,\bigl|u(y)\bigr| \;=\; 0.
	\]
	By Lemma \ref{stupbound}, we get
	\[
	\begin{aligned}
		\frac{1}{\mu(x)}\sum_{y\neq x}W_s(x,y)-1=&\frac{s}{\Gamma(1-s)}\,
		\int_0^{1}\left(1-\mu(x)p(t,x,x)\right)\,t^{-1-s}\,dt\\+&\frac{s}{\Gamma(1-s)}\,
		\int_1^{\infty}\left(1-\mu(x)p(t,x,x)\right)\,t^{-1-s}\,dt-1.
	\end{aligned}
	\]
	Note that
	\[\begin{aligned}
		&	\frac{s}{\Gamma(1-s)}\,
		\int_1^{\infty}\left(1-\mu(x)p(t,x,x)\right)\,t^{-1-s}\,dt-1\\=&-\frac{s}{\Gamma\left(1-s\right)}\int_{1}^{\infty}\mu(x)p(t,x,x)t^{-1-s}dt+\frac{s}{\Gamma(1-s)}\int_1^{\infty}t^{-1-s}dt-1\\=&-\frac{s}{\Gamma\left(1-s\right)}\int_{1}^{\infty}\mu(x)p(t,x,x)t^{-1-s}dt+\frac{1}{\Gamma(1-s)}-1.
	\end{aligned}\]
	Again, by \eqref{gaussian}, we obtain that 
	\[\lim\limits_{s\rightarrow 0^+}\left|\frac{1}{\mu(x)}\sum_{y\ne x}W_s(x,y)\,
	-1\right|=0.\]
	Therefore, for every \(u\in C_c(\mathbb{Z}^d)\), 
	\[
	\lim_{s\to0^+}(-\Delta)^s u(x) = u(x),\:\:x\in \mathbb{Z}^d.
	\]
\qed

We have already given an pointwise expression for the logarithmic Laplacian on infinite graphs via the Bochner integral formula, see Theorem \ref{pointlog11}.  We now exploit the fact that \(\log\left(-\Delta\right)\) appears as the first‐order term in the Taylor expansion of \((-\Delta)^s\) at \(s=0\) 
	\[
\left(\log\left(-\Delta\right)u\right)(x):=\left.\frac{d}{ds}\right|_{s=0}(-\Delta)^s u(x),
\quad x\in \mathbb{Z}^d.
\]
to derive an alternative pointwise representation proof.

\begin{theorem}\label{pointwise}
	Let \(u\in C_c(\mathbb{Z}^d)\). For each \(x\in\mathbb{Z}^d\), one has the pointwise representation
\begingroup\small	\[
\left(\log\left(-\Delta\right)u\right)(x)
	=\frac{1}{\mu(x)}
	\sum_{\substack{y\in\mathbb{Z}^d,\:y\neq x}}
	W_{\log}(x,y)\bigl(u(x)-u(y)\bigr)
	\;-\;
	\frac{1}{\mu(x)}
	\sum_{y\in\mathbb{Z}^d}
	W(x,y)\,u(y)
	\;+\;
	\Gamma'(1)\,u(x),
	\]\endgroup
	where the kernels
	\[
	W_{\log}(x,y)
	:=\mu(x)\,\mu(y)\int_{0}^{1}p(t,x,y)\,t^{-1}\,dt,
	\quad
	W(x,y)
	:=\mu(x)\,\mu(y)\int_{1}^{\infty}p(t,x,y)\,t^{-1}\,dt
	\]
	are symmetric and positive.
\end{theorem}

\begin{proof}
	Fix \(x\in\mathbb{Z}^d\) and \(u\in C_c(\mathbb{Z}^d)\).  We write
	\[
	\frac{(-\Delta)^s u(x)-u(x)}{s}
	=\frac{1}{s}\Bigl[\frac{s}{\Gamma(1-s)}
	\sum_{y\ne x}\mu(y)\int_0^\infty p(t,x,y)\,t^{-1-s}dt\,(u(x)-u(y)) - u(x)\Bigr].
	\]
	Split the sum into two terms:
	\[
	\frac{(-\Delta)^s u(x)-u(x)}{s}
	\;=\;
	A_s(x)u(x)\;-\;B_s(x),
	\]
	where
	\[
	A_s(x)
	=\frac{1}{s}\Biggl[
	\frac{s}{\Gamma(1-s)}
	\sum_{y\neq x}\mu(y)\int_{0}^{\infty}p(t,x,y)\,t^{-1-s}dt
	\;-\;1
	\Biggr],
	\]
	and by Lemma \ref{stupbound} and dominated convergence this becomes
	\[
	A_s(x)
	=\frac{1}{s}\Biggl[
	\frac{s}{\Gamma(1-s)}
	\int_{0}^{\infty}\bigl(1-\mu(x)\,p(t,x,x)\bigr)\,t^{-1-s}dt
	\;-\;1
	\Biggr].
	\]
	The other term is
	\[
	B_s(x)
	=\frac{1}{\Gamma(1-s)}
	\sum_{y\neq x}\mu(y)\,u(y)
	\int_{0}^{\infty}p(t,x,y)\,t^{-1-s}dt.
	\]
	Note that
	\[
	\int_0^\infty p(t,x,y)\,t^{-1-s}dt
	=\int_0^1 p(t,x,y)\,t^{-1-s}dt
	+\int_1^\infty p(t,x,y)\,t^{-1-s}dt.
	\]
	On \([0,1]\), by \eqref{444} when $s$ is sufficiently small
	\[
	p(t,x,y)\,t^{-1-s}\lesssim
	t^{\frac{d(x,y)}2-1},
	\]
	which is integrable.  On \([1,\infty)\), by \eqref{gaussian}
	\begin{equation}\label{dsjgj}
		p(t,x,y)\,t^{-1-s}\lesssim t^{-1-\frac{d}{2}},
	\end{equation}
	also integrable.  So the dominated convergence theorem gives
	\[
	\lim_{s\to0^+}\int_0^\infty p(t,x,y)\,t^{-1-s}\,dt
	=\int_0^\infty p(t,x,y)\,t^{-1}\,dt, \:\:x\ne y.
	\]
	Since \(u\) has compact support, the sum over \(y\neq x\) is finite. Hence as \(s\to0^+\),
	\begin{equation}\label{pp1}
		B_s(x)
		\rightarrow
		\sum_{y\neq x}\mu(y)\,u(y)
		\int_0^\infty p(t,x,y)\,t^{-1}dt.
	\end{equation}

	Next, we consider $A_{s}\left(x\right).$ By Lemma \ref{stupbound}, we have
\begin{equation}
	\begin{aligned}\label{asx}
		A_s\left(x\right)=&\frac{1}{\Gamma\left(1-s\right)}\int_{0}^{1}\bigl(1-\mu(x)\,p(t,x,x)\bigr)\,t^{-1-s}dt-\frac{1}{\Gamma\left(1-s\right)}\int_{1}^{\infty}\mu(x)p(t,x,x)t^{-1-s}dt\\+&\frac{1}{s}\left(\frac{s}{\Gamma(1-s)}\int_1^{\infty}t^{-1-s}dt-1\right)
	\end{aligned}
\end{equation}
	The last term simplifies using \(\int_1^\infty t^{-1-s}\,dt = 1/s\), yielding
	\begin{equation} \label{pp2}
		\frac{1}{s} \left( \frac{1}{\Gamma(1-s)} - 1 \right) \xrightarrow[]{s \to 0^+} -\gamma = \Gamma'(1).
	\end{equation}
	For the second term, the estimate \eqref{dsjgj} implies integrability of the integrand, so the dominated convergence theorem yields
	\begin{equation}\label{pp3}
		-\frac{1}{\Gamma(1-s)}\int_{1}^{\infty}\mu(x)\,p(t,x,x)\,t^{-1-s}\,dt
		\;\xrightarrow[]{s\to0^+}\;
		-\int_{1}^{\infty}\mu(x)\,p(t,x,x)\,t^{-1}\,dt.
	\end{equation}
	For the first term, using the expansion
	\[
	\mu(x) p(t,x,x) = 1 + t\mu(x)\,p'(\tau,x,x), \quad \tau \in (0,1),
	\]
	and by Levi lemma and Lemma \ref{higher‐derivative-nearest}
	\begin{equation}
		\begin{aligned}\label{pp4}
			\frac{1}{\Gamma(1-s)}\int_{0}^{1}\bigl(1-\mu(x)\,p(t,x,x)\bigr)\,t^{-1-s}\,dt
			\;\xrightarrow[]{s\to0^+}\;
			\int_{0}^{1}\bigl(1-\mu(x)\,p(t,x,x)\bigr)\,t^{-1}\,dt,
		\end{aligned}
	\end{equation}
	Combining \eqref{pp1}, \eqref{pp2}, \eqref{pp3}, and \eqref{pp4}, we conclude that
	\[
	\begin{aligned}
		L_{\Delta} u(x)
		&= -\sum_{y \ne x} \mu(y)\,u(y) \int_0^\infty p(t,x,y)\,t^{-1} \,dt
		+ \Gamma'(1)\,u(x) \\
		&\quad - \mu(x)u(x) \int_1^\infty \,p(t,x,x)\,t^{-1} \,dt
		+ u(x) \int_0^1 \bigl(1 - \mu(x)\,p(t,x,x)\bigr)\,t^{-1} \,dt.
	\end{aligned}
	\]
	By Lemma \ref{stupbound}, this matches the representation
	\[
	L_{\Delta} u(x)
	= \frac{1}{\mu(x)} \sum_{y \ne x} W_{\log}(x,y)\,(u(x) - u(y))
	- \frac{1}{\mu(x)} \sum_{y \in \mathbb{Z}^d} W(x,y)\,u(y)
	+ \Gamma'(1)\,u(x),
	\]
	with
	\[
	W_{\log}(x,y) := \mu(x)\mu(y) \int_0^1 p(t,x,y)\,t^{-1} \,dt,
	\quad
	W(x,y) := \mu(x)\mu(y) \int_1^\infty p(t,x,y)\,t^{-1} \,dt.
	\]
\end{proof}

Finally, we establish the integrability of the logarithmic Laplacian and prove stronger \(\ell^p\)-convergence, analogous to the corresponding results in \(\mathbb{R}^n\), see \cite[Theorem 1.1]{chen2019dirichlet}.

	\vspace{1\baselineskip}
	
\noindent \textbf{Proof of Theorem \ref{thm:lplimit11}:}

	Recall from Theorem \ref{pointwise} that, for each \(x\),
	\[
	L_{\Delta}u(x)
	= A(x) - B(x) + \Gamma'(1)\,u(x),
	\]
	where
	\[
	A(x)
	=\frac{1}{\mu(x)}
	\sum_{y\neq x}W_{\log}(x,y)\bigl(u(x)-u(y)\bigr),
	\quad
	B(x)
	=\frac{1}{\mu(x)}
	\sum_{y}W(x,y)\,u(y).
	\]
	
	\medskip\noindent\emph{Step 1: \(L_{\Delta}u\in\ell^p\).}  Since \(u\) is finitely supported, \(\Gamma'(1)\,u\in\ell^p\) trivially.  
	
	For the “long‐range’’ term \(B\), Proposition \ref{prop:Wxy11} gives
	\[
	W(x,y)\lesssim d(x,y)^{-d},
	\]
	so
	\[
	|B(x)|
	\;\lesssim\;
	\sum_{y}|u(y)|\,\mu(y)\,d(x,y)^{-d}
	\;+\;|u(x)|.
	\]
	Setting \(g_y(x)=d(x,y)^{-d}\) (zero at \(x=y\)) and using that \(g_y\in\ell^p\) if and only if \(p>1\), one sees the finite support of \(u\) forces \(B\in\ell^p\) for all \(1<p\le\infty\).  
	
	For the “log‐short’’ term \(A\), Proposition \ref{prop:Wlog11} shows
	$$\,W_{\log}(x,y)\lesssim d(x,y)^{-1}(e/d(x,y))^{d(x,y)},$$
	whence by a similar shell‐counting argument
	\(\sum_{y\neq x}W_{\log}(x,y)\in\ell^\infty\).  Since \(u\) is finitely supported, a finite linear combination of such translates lies in every \(\ell^p\). 
	
\medskip\noindent\emph{Step 2: Convergence of the difference quotient.}  
 To upgrade this to \(\ell^p\)-convergence, we exhibit a single nonnegative function \(F\in\ell^p(\mathbb{Z}^d),\:1<p\le \infty\) with
\[
\biggl|\frac{(-\Delta)^s u(x)-u(x)}{s}\biggr|
\;\le\;
F(x)
\quad\text{for all sufficiently small }s>0.
\]

From the representation in the proof of Theorem \ref{pointwise} one sees that
\[
\frac{(-\Delta)^s u(x)-u(x)}{s}
= A_s(x)u(x) - B_s(x),
\]
where
\[
	A_s(x)
=\frac{1}{s}\Biggl[
\frac{s}{\Gamma(1-s)}
\int_{0}^{\infty}\bigl(1-\mu(x)\,p(t,x,x)\bigr)\,t^{-1-s}dt
\;-\;1
\Biggr],
\]
\[
B_s(x)
=\frac{1}{\Gamma\left(1-s\right)}\sum_{y\ne x} \mu(y)u(y)\!\int_0^\infty p(t,x,y)\,t^{-1-s}\,dt.
\]
Similar to the proof of Proposition \ref{prop:Wxy11}, we obtain that
\[\begin{aligned}
	\left|B_{s}\left(x\right)\right|\lesssim&  \sum_{y\ne x}u\left(y\right)\int_0^{\infty}p\left(t,x,y\right)t^{-1-s}dt \\ \lesssim& \sum_{y\ne x}u\left(y\right)\left(\int_0^{1}p\left(t,x,y\right)t^{-1-s}dt+\int_1^{\infty}p\left(t,x,y\right)t^{-1}dt\right)\\ \lesssim &  \sum_{y\ne x}u\left(y\right)d\left(x,y\right)^{-d}.
\end{aligned}\]

Next, we analysis $A_s\left(x\right).$ By (\ref{asx}) and (\ref{dsjgj}), we have
\[\begin{aligned}
	A_s\left(x\right)\lesssim&\int_{0}^{1}\bigl(1-\mu(x)\,p(t,x,x)\bigr)\,t^{-1-s}dt+\int_{1}^{\infty}\mu(x)p(t,x,x)t^{-1-s}dt+1\\ \lesssim &\int _0^1 t^{-s}dt+\int_1^{\infty}t^{-\frac{d}{2}-1}dt+1\le c
\end{aligned}\]
Thus, $A_s \in \ell^{\infty}.$ Set $F\left(x\right)=cu(x)+c \sum_{y\ne x}u\left(y\right)d\left(x,y\right)^{-d}\in \ell^p$ for $1<p\le \infty.$ By the dominated convergence theorem, for $1<p<\infty$
\[\lim\limits_{s\rightarrow 0^+}\sum_{x\in \mathbb{Z}^d}\left| \frac{(-\Delta)^{s}u(x)-u(x)}{s}-L_{\Delta}u\left(x\right)\right|^p=0.\]
Since $||u||_{\ell^\infty}\lesssim ||u||_{\ell^p}$,
\[\lim\limits_{s\rightarrow 0^+}||\frac{(-\Delta)^{s}u-u}{s}-L_{\Delta}u||_{\infty}=0.\]

\subsection{Fourier Transform and Diffusion Kernels}
In order to apply the discrete Fourier transform, we work on the standard lattice graph \(\mathbb{Z}^d\) equipped with unit vertex measure and unit edge weights. In this translation‐invariant setting the Laplacian and its fractional powers are diagonalized by plane waves.

 For any \(u\in \ell^{1}(\mathbb{Z}^d)\), define the unitary Fourier transform \(\widehat u:[-\pi,\pi]^d\to\mathbb{C}\) by  
\[
\widehat u(\xi)
=\frac{1}{(2\pi)^{d/2}}
\sum_{x\in\mathbb{Z}^d}u(x)e^{-i\langle x,\xi\rangle},
\quad
\xi\in[-\pi,\pi]^d,
\]
where \(\langle x,\xi\rangle=\sum_{j=1}^d x_j\,\xi_j\). Thus, $\widehat u\in \ell^{\infty}\left([-\pi,\pi]^d\right).$ The inversion formula is then
\[
u(x)
=\frac{1}{(2\pi)^{d/2}}
\int_{[-\pi,\pi]^d}\widehat u(\xi)\,e^{i\langle x,\xi\rangle}\,d\xi,
\]

We now define $(-\Delta)^s$ for $0<s<1$ via the Bochner integral formula:
\begin{equation}\label{bochfrac11}
	(-\Delta)^{s}u=\frac{s}{\Gamma(1-s)}
	\int_{0}^{\infty}\bigl(u-e^{t\Delta}u\bigr)\,t^{-1-s}\,dt,u\in \ell^{\infty}(\mathbb{Z}^d).
\end{equation}

 Let
$$
\Phi(\xi)=\sum_{i=1}^d\bigl(2-2\cos\xi_i\bigr).
$$
It is not hard to see that $\left(-\Delta\right)^{s}u \in \ell^1(\mathbb{Z}^d)$ if $u\in\ell^1(\mathbb{Z}^d)$. Then, as shown in \cite{wang2023eigenvalue}, for every $u\in\ell^1(\mathbb{Z}^d)$ and all $\xi\in[-\pi,\pi]^d$,
$$
\widehat{(-\Delta)^s u}(\xi)
=\Phi(\xi)^s\,\widehat u(\xi).
$$
In other words, in Fourier space the fractional Laplacian is simply the multiplier $\Phi(\xi)^s$.

\begin{proposition}
	Let \(u\in C_c(\mathbb{Z}^d)\), then the Fourier transform of the logarithmic Laplacian satisfies
	\[
	\widehat{\log\left(-\Delta\right)u}(\xi)
	=\ln\!\bigl(\Phi(\xi)\bigr)\,\widehat u(\xi)\quad \text{in}\:\:\ell^2([-\pi,\pi]^d)
	\quad \xi\in[-\pi,\pi]^d.
	\]
\end{proposition}

\begin{proof}
	By Theorem \ref{thm:lplimit11}, 
	\[
	\log\left(-\Delta\right)u
	=\lim_{s\to0^+}\frac{(-\Delta)^s u - u}{s}
	\]
	in \(\ell^p,\:1<p\le \infty\) (in particular in \(\ell^2\)), and the discrete Fourier transform
	\(\mathcal F:\ell^2(\mathbb{Z}^d)\to \ell^2([-\pi,\pi]^d)\) is unitary.  Hence
	\[
	\widehat{\log\left(-\Delta\right)u}
	=\lim_{s\to0^+}
	\frac{\widehat{(-\Delta)^s u}-\widehat u}{s}
	\;\;\text{in }\ell^2.
	\]
	But for each fixed \(s\in(0,1)\),
	\[
	\widehat{(-\Delta)^s u}(\xi)
	=\Phi(\xi)^s\,\widehat u(\xi),\:\text{in}\:\:\ell^2
	\]
	so
	\[
	\frac{\widehat{(-\Delta)^s u}(\xi)-\widehat u(\xi)}{s}
	=\frac{\Phi(\xi)^s-1}{s}\,\widehat u(\xi).
	\]
Thus,
	\[
	\frac{\Phi(\xi)^s-1}{s}\;\longrightarrow\;\ln\!\bigl(\Phi(\xi)\bigr)
	\quad s\to0^+
	\]
yields
	\[
	\widehat{\log\left(-\Delta\right)u}(\xi)
	=\ln\!\bigl(\Phi(\xi)\bigr)\,\widehat u(\xi),
	\]
	as claimed.
\end{proof}

	Since the fractional Laplacian acts by multiplication:
\(\widehat{(-\Delta)^s u}(\xi)=\Phi(\xi)^s\widehat u(\xi)\), the semigroup \(e^{-t(-\Delta)^s}\) has symbol \(e^{-t\,\Phi(\xi)^s}\), and its kernel is
\[
p_s(t,x,y)
=\bigl(e^{-t(-\Delta)^s}\delta_y\bigr)(x)
=\frac{1}{(2\pi)^d}
\int_{[-\pi,\pi]^d}
e^{-t\,\Phi(\xi)^s}\,e^{\,i(x-y)\cdot\xi}\,d\xi.
\]

We next collect some fundamental properties of the fractional diffusion kernel \(p_s(t,x,y)\).

\begin{remark}
For each fixed \(0<s<1\), the following hold:
	\begin{enumerate}
		\item[(i)] \emph{Symmetry and Positivity:} 
		\(p_s(t,x,y)=p_s(t;y,x)>0\) for all \(t>0\) and \(x,y\in \mathbb{Z}^d\).
		\item[(ii)] \emph{Smoothness in Time:} 
		\(t\mapsto p_s(t,x,y)\in C^\infty([0,\infty))\).
		\item[(iii)] \emph{Fractional Heat Equation:}
		\(\partial_t p_s +(-\Delta)^s_x p_s=0\).
		\item[(iv)] \emph{Semigroup Property:}
		\(p_s(t_1+t_2;x,y)=\sum_{z\in \mathbb{Z}^d}p_s(t_1;x,z)p_s(t_2;z,y)\).
		\item[(v)] \emph{Contraction:}
		\(\sum_{y}p_s(t,x,y)\le1\) for all \(t\ge0\).
		\item[(vi)] \emph{Delta‐Limit as \(t\to0\):}
		\(p_s(t,x,y)\to\delta_{x,y}\) on \(\mathbb{Z}^d\).
	\end{enumerate}
\end{remark}

Now, we give the proof of Proposition \ref{prop:large_time11} and Proposition \ref{prop:tail_asymptotic11}. 
	\vspace{1\baselineskip}
	
\noindent \textbf{Proof of Proposition \ref{prop:large_time11}:}

	Recall the Fourier representation
	\[
	p_s(t,x,y)
	=\frac{1}{(2\pi)^d}
	\int_{[-\pi,\pi]^d}e^{-t\,\Phi(\xi)^s}\,e^{\,i(x-y)\cdot\xi}\,d\xi,
	\]
	where \(\Phi(\xi)=\sum_{j=1}^d(2-2\cos\xi_j)\ge0\) and \(\Phi(\xi)=|\xi|^2+o\left(|\xi|^2\right)\) as \(\left|\xi\right|\to0\).
	
	\medskip\noindent\textbf{(i) Uniform decay.}
	Dropping the oscillatory factor,
	\[
	p_s(t,x,y)\le\frac{1}{(2\pi)^d}
	\int_{[-\pi,\pi]^d}e^{-t\,\Phi(\xi)^s}\,d\xi.
	\]
Choose \(\delta>0\) small enough that 
\(\Phi(\xi)\ge\tfrac12|\xi|^2\) for all \(|\xi|\le\delta\).  Since \(\Phi\) is continuous and \(\Phi(\xi)>0\) on the compact set \(\{|\xi|\ge\delta\}\), there is \(C_0>0\) so that \(\Phi(\xi)\ge C_0\) whenever \(|\xi|\ge\delta\).  We split
\[
\int_{[-\pi,\pi]^d}e^{-t\,\Phi(\xi)^s}\,d\xi
=\int_{|\xi|\le\delta}e^{-t\,\Phi(\xi)^s}\,d\xi
+\int_{\{\xi\in[-\pi,\pi]^d:|\xi|>\delta\}}e^{-t\,\Phi(\xi)^s}\,d\xi.
\]
On \(|\xi|\le\delta\), note that
\[
e^{-t\,\Phi(\xi)^s}
\le e^{-t\,( \tfrac12|\xi|^2 )^s}
=e^{-t\,2^{-s}|\xi|^{2s}},
\]
so
\[
\int_{|\xi|\le\delta}e^{-t\,\Phi(\xi)^s}\,d\xi
\le\int_{\mathbb{R}^d}e^{-t\,2^{-s}|\xi|^{2s}}\,d\xi.
	\]
Set
$$
\eta = (t\,2^{-s})^{1/(2s)}\,\xi,
$$
then
$$
\int_{\mathbb{R}^d}e^{-t\,2^{-s}|\xi|^{2s}}\,d\xi
=\int_{\mathbb{R}^d}e^{-|\eta|^{2s}}
\,(t\,2^{-s})^{-d/(2s)}\,d\eta\sim t^{-\,\tfrac{d}{2s}}, \quad t>0.
$$

On \(|\xi|>\delta\), since
\[
e^{-t\,\Phi(\xi)^s}
\le e^{-t\,C_0^s},
\]
\[
\int_{\{\xi\in[-\pi,\pi]^d:|\xi|>\delta\}}e^{-t\,\Phi(\xi)^s}\,d\xi
\le(2\pi)^d\,e^{-C_0^s\,t}\sim e^{-C_0^{s}t},\quad t>0\]
which is negligible compared to \(t^{-d/(2s)}\).  

Combining these estimates gives
\[
\int_{[-\pi,\pi]^d}e^{-t\,\Phi(\xi)^s}\,d\xi
\;\lesssim\;
t^{-d/(2s)},
\]
and therefore \(\sup_{x,y}p_s(t,x,y)\lesssim t^{-d/(2s)}\).  

	\medskip\noindent\textbf{(ii) Asymptotic constant.}
	Fix a finite set of relative positions \(h=x-y\).  Split the integral into
	\[
	I_1=\int_{|\xi|\le R}\quad\text{and}\quad I_2=\int_{|\xi|>R}
	\]
where \(R>0\) is sufficiently small. Then: 
	
	(i) On \(|\xi|>R\), \(\Phi(\xi)^s\ge C>0\), so
	\[
	|I_2|\le \int_{|\xi|>R}e^{-tC^{s}}\,d\xi=\;o\bigl(t^{-d/(2s)}\bigr),\quad t\rightarrow \infty.
	\]
	
(ii) On \(|\xi|\le R\), set \(\eta = t^{1/(2s)}\,\xi\).  Then
	\[
	I_1
	=\int_{|\xi|\le R}e^{-t\,\Phi(\xi)^s}e^{i\,h\cdot\xi}\,d\xi
	=t^{-d/(2s)}\int_{|\eta|\le t^{1/(2s)}R}e^{-t\left(\Phi(t^{-\frac{1}{2s}}\eta)\right)^s}e^{i\,h\cdot \eta\,t^{-1/(2s)}}\,d\eta.
	\]
	As \(t\to\infty\), the phase \(h\cdot\eta\,t^{-1/(2s)}\to0\).  By dominated convergence,
	\[
	I_1
	\sim t^{-d/(2s)}
	\int_{\mathbb{R}^d}e^{-|\eta|^{2s}}\,d\eta,\quad t\rightarrow\infty.
	\]
	Combining contributions,
	\[
	p_s(t,x,y)
	=\frac{1}{(2\pi)^d}(I_1+I_2)
	\sim\frac{1}{(2\pi)^d}\Bigl(\int_{\mathbb{R}^d}e^{-|\eta|^{2s}}\,d\eta\Bigr)\,
	t^{-d/(2s)}=\frac{\pi^{-d/2}}{s2^{d}\Gamma(\tfrac d2)}\;\Gamma\!\Bigl(\tfrac{d}{2s}\Bigr) t^{-d/(2s)}.
	\]  This completes the proof. \qed

\vspace{1\baselineskip}

\noindent \textbf{Proof of Proposition \ref{prop:tail_asymptotic11}:}

Set $\mathbb{T}^d=[-\pi,\pi]^d, k=x-y$ and write
	\[
	f(\xi)=e^{-t\,\Phi(\xi)^s},\qquad
	\Phi(\xi)=\sum_{j=1}^d\bigl(2-2\cos\xi_j\bigr).
	\]
	Choose \(0<\delta\ll1\) and a radial cutoff 
	\(\chi\in C_c^\infty(\mathbb{T}^d)\) with
	\[
	\chi(\xi)=
	\begin{cases}
		1,&|\xi|\le\tfrac12,\\
		0,&|\xi|\ge1.
	\end{cases}
	\]
	 Decompose
	\begingroup\small\[
	p_s(t,x,y)=\frac{1}{(2\pi)^d}
	\int_{\mathbb{T}^d}
\left(f(\xi)-1\right)\chi\left(\frac{\xi}{\delta}\right)e^{\,ik\cdot\xi}\,d\xi+\frac{1}{(2\pi)^d}
\int_{\mathbb{T}^d}
\left(f(\xi)-1\right)\left(1-\chi\left(\frac{\xi}{\delta}\right)\right)e^{\,ik\cdot\xi}\,d\xi
	\]\endgroup
	{\bf (i) Smooth remainder.}  On \(\{|\xi|\ge\tfrac\delta2\}\),
	\(\Phi(\xi)\ge c>0\), so 
	\[\left(f(\xi)-1\right)\left(1-\chi\left(\frac{\xi}{\delta}\right)\right)\in C^\infty(\mathbb{T}^d)\] 
	and is $2\pi$-periodic function. Repeated integration by parts shows \cite[Theorem 3.2.9]{grafakos2008classical}
	\[
\int_{\mathbb{T}^d}
\left(f(\xi)-1\right)\left(1-\chi\left(\frac{\xi}{\delta}\right)\right)e^{\,ik\cdot\xi}\,d\xi
	=O\bigl(|k|^{-N}\bigr),
	\quad\forall N.
	\]
\textbf{(ii) Principal part.}  Set $	r = |k|, \omega = \frac{k}{r},$ make the change of variables \(\eta = r\,\xi\).  Then
		\[
		r^{d+2s}	p_s(t,x,y)=\frac{1}{(2\pi)^d}I_1(k)+O\bigl(|k|^{-N}\bigr)\]
		where 
		\[
	\begin{aligned}
			I_1(k):=& r^{2s+d}
		\int_{|\xi|\le \delta}
	\left(f(\xi)-1\right)\chi\left(\frac{\xi}{\delta}\right)e^{\,ik\cdot\xi}\,d\xi
		\\=& r^{2s}
		\int_{|\eta|\le r\delta}
		\Bigl(f\bigl(\tfrac\eta r\bigr)-1\Bigr)\chi\left(\frac{\eta}{r\delta}\right)\,
		e^{\,i\,\omega\cdot\eta}\,
		d\eta.
	\end{aligned}
		\]
By the Taylor expansion,
\[f\left(\xi\right)-1=-t|\xi|^{2s}+O\left(|\xi|^{3s}\right),|\xi|\le \delta.\]
	 Hence,
	 \[I_{1}\left(k\right)=-t	\int_{|\eta|\le r\delta}
	|\eta|^{2s}\chi\left(\frac{\eta}{r\delta}\right)\,
	 e^{\,i\,\omega\cdot\eta}\,
	 d\eta+r^{-s}\int_{|\eta|\le r\delta}
	 O\left(|\eta|^{3s}\right)\chi\left(\frac{\eta}{r\delta}\right)\,
	 e^{\,i\,\omega\cdot\eta}\,
	 d\eta.\]
	 Completely analogous to the analysis in Lemma \ref{jieduan11}, there exists $C>0$ independent of $r$ such that
	 \[\left|\int_{|\eta|\le r\delta}
	 O\left(|\eta|^{3s}\right)\chi\left(\frac{\eta}{r\delta}\right)\,
	 e^{\,i\,\omega\cdot\eta}\,
	 d\eta\right|\le C.\]
	 Therefore, 
	\[
	\lim\limits_{|k|\rightarrow \infty}I_1(k)\;\longrightarrow\;
	-\,t\lim\limits_{N\rightarrow \infty}
	\int_{\mathbb{R}^d}|\eta|^{2s}\chi\left(\frac{\eta}{N}\right)\,e^{i\omega\cdot\eta}\,d\eta=-tA_{s,d},
	\]
	which is the desired result.\qed

The logarithmic Laplacian \(\log(-\Delta)\) corresponds to multiplication by \(\ln\Phi(\xi)\) in the Fourier domain.  Hence the log‐heat semigroup
\[
e^{-t\,\log(-\Delta)}
\]
acts by \(\exp\bigl(-t\ln\Phi(\xi)\bigr)=\Phi(\xi)^{-t}\).  By Fourier inversion, its formal diffusion kernel is
\[
p_{\log}(t,x,y)
=\bigl(e^{-t\,\log(-\Delta)}\delta_{y}\bigr)(x)
=\frac{1}{(2\pi)^d}
\int_{[-\pi,\pi]^d}
\Phi(\xi)^{-t}\,e^{\,i\,(x-y)\cdot\xi}\,d\xi.
\]

The only singularity of the integrand 
\(\Phi(\xi)^{-t}\) on \([-\pi,\pi]^d\) occurs at \(\xi=0\), where 
\[
\Phi(\xi)=\sum_{j=1}^d(2-2\cos\xi_j)\sim|\xi|^2
\quad(\xi\to0).
\]
Hence near \(\xi=0\) the integrand behaves like 
\(|\xi|^{-2t}\), and the integral 
\[
\int_{|\xi|<\delta}|\xi|^{-2t}\,d\xi
\sim \int_0^\delta r^{d-1-2t}\,dr
\]
converges precisely when $d-1-2t>-1\Longleftrightarrow
t<\frac d2.$
Thus, for each fixed \(x,y\) the log‐diffusion kernel integral
\[
p_{\log}(t,x,y)
=\frac1{(2\pi)^d}
\int_{[-\pi,\pi]^d}
\Phi(\xi)^{-t}\,e^{i(x-y)\cdot\xi}\,d\xi
\]
is (absolutely) convergent for $0\;\le\;t\;<\;\frac d2$
and 
\[p_{\log}(0;x,y)=\delta_{x,y};\: p_{\log}(t;x,x)=\frac1{(2\pi)^d}
\int_{[-\pi,\pi]^d}
\Phi(\xi)^{-t}d\xi<\infty.\]

In the next step, we give the proof of Proposition \ref{prop:log_blowup_and_delta11} and Proposition \ref{prop:log_offdiag_asym11}.

\vspace{1\baselineskip} 

\noindent \textbf{Proof of Proposition \ref{prop:log_blowup_and_delta11}:}

	Set $k=x-y,$ write
	\[
	p_{\log}(t,x,y)
	=\frac1{(2\pi)^d}\int_{[-\pi,\pi]^d}\Phi(\xi)^{-t}e^{i (x-y)\cdot\xi}\,d\xi
	=\frac1{(2\pi)^d}\Bigl(I_1+I_2\Bigr),
	\]
	where
	\[
	I_1=\int_{|\xi|\le\delta}\Phi(\xi)^{-t}e^{i k\cdot\xi}\,d\xi,
	\quad
	I_2=\int_{|\xi|>\delta}\Phi(\xi)^{-t}e^{i k\cdot\xi}\,d\xi,
	\]
	with a small fixed \(\delta>0\). Since \(I_2\) remains bounded as \(t\to d/2\), hence
	 \((d-2t)\,I_2\to0\).
	 
	  Next we consider $I_1.$ Set $\xi=\eta\left(d-2t\right)^{\alpha},\alpha>0,$ then
	 \[I_1=\left(d-2t\right)^{d\alpha}\int_{|\eta|\le \left(d-2t\right)^{-\alpha}\delta}\Phi(\eta\left(d-2t\right)^{\alpha})^{-t}e^{ik\cdot \eta\left(d-2t\right)^{\alpha}}d\eta.\]
	 Therefore, 
	 \[\begin{aligned}
	 	\lim\limits_{t\rightarrow \frac{d}{2}}\left(d-2t\right)I_1=&	\lim\limits_{t\rightarrow \frac{d}{2}}\left(d-2t\right)^{d\alpha+1}\int_{|\eta|\le \left(d-2t\right)^{-\alpha}\delta}\Phi(\eta\left(d-2t\right)^{\alpha})^{-t}e^{ik\cdot \eta\left(d-2t\right)^{\alpha}}d\eta\\=&	\lim\limits_{t\rightarrow \frac{d}{2}}\left(d-2t\right)^{(d-2t)\alpha+1}\int_{|\eta|\le \left(d-2t\right)^{-\alpha}\delta}|\eta|^{-2t}d\eta\\=&|\mathbb{S}^{d-1}|\lim\limits_{t\rightarrow \frac{d}{2}}\left(d-2t\right)^{(d-2t)\alpha+1}\int_{0}^ {\left(d-2t\right)^{-\alpha}\delta}r^{d-1-2t}dr\\=&|\mathbb{S}^{d-1}|\lim\limits_{t\rightarrow \frac{d}{2}}\left(d-2t\right)^{(d-2t)\alpha}\left(\left(d-2t\right)^{-\alpha}\delta\right)^{d-2t}=|\mathbb{S}^{d-1}|,
	 \end{aligned}\]
	 which is the desired result.\qed

\vspace{1\baselineskip} 

\noindent \textbf{Proof of Proposition \ref{prop:log_offdiag_asym11}:}

Set $k=x-y$ and choose \(0<\delta\ll1\) and a radial cutoff 
\(\chi\in C_c^\infty(\mathbb{T}^d)\) with
\[
\chi(\xi)=
\begin{cases}
	1,&|\xi|\le\tfrac12,\\
	0,&|\xi|\ge1.
\end{cases}
\]
Decompose
\begingroup\small\[
p_{\log}(t,x,y)=\frac{1}{(2\pi)^d}
\int_{\mathbb{T}^d}
\Phi(\xi)^{-t}\chi\left(\frac{\xi}{\delta}\right)e^{\,ik\cdot\xi}\,d\xi+\frac{1}{(2\pi)^d}
\int_{\mathbb{T}^d}
\Phi(\xi)^{-t}\left(1-\chi\left(\frac{\xi}{\delta}\right)\right)e^{\,ik\cdot\xi}\,d\xi
\]\endgroup
{\bf (i) Smooth remainder.}  On \(\{|\xi|\ge\tfrac\delta2\}\),
\(\Phi(\xi)\ge c>0\), so 
\[\Phi(\xi)^{-t}\left(1-\chi\left(\frac{\xi}{\delta}\right)\right)\in C^\infty(\mathbb{T}^d)\] 
and is $2\pi$-periodic function. Repeated integration by parts shows \cite[Theorem 3.2.9]{grafakos2008classical}
\[
\int_{\mathbb{T}^d}
\Phi(\xi)^{-t}\left(1-\chi\left(\frac{\xi}{\delta}\right)\right)e^{\,ik\cdot\xi}\,d\xi
=O\bigl(|k|^{-N}\bigr),
\quad\forall N.
\]
\textbf{(ii) Principal part.}  Set $	r = |k|, \omega = \frac{k}{r},$ make the change of variables \(\eta = r\,\xi\).  Then
\[
r^{d-2t}	p_{\log}(t,x,y)=\frac{1}{(2\pi)^d}I(k)+O\bigl(|k|^{-N}\bigr)\]
where 
	\begingroup\small\[
\begin{aligned}
	I(k):=r^{d-2t}
	\int_{|\xi|\le \delta}
\Phi(\xi)^{-t}\chi\left(\frac{\xi}{\delta}\right)e^{\,ik\cdot\xi}\,d\xi= r^{-2t}
	\int_{|\eta|\le r\delta}
\Phi(\frac{\eta}{r})^{-t}\chi\left(\frac{\eta}{r\delta}\right)\,
	e^{\,i\,\omega\cdot\eta}\,
	d\eta.
\end{aligned}
\]\endgroup
By the Taylor expansion,
\[\Phi(\xi)^{-t}=|\xi|^{-2t}\left(1+O\left(1\right)\right),|\xi|\le \delta.\]
Hence,
\[I\left(k\right)=	\int_{|\eta|\le r\delta}
|\eta|^{-2t}\chi\left(\frac{\eta}{r\delta}\right)\,
e^{\,i\,\omega\cdot\eta}\,
d\eta+r^{-2t}\int_{|\eta|\le r\delta}
O\left(1\right)\chi\left(\frac{\eta}{r\delta}\right)\,
e^{\,i\,\omega\cdot\eta}\,
d\eta.\]
Completely analogous to the analysis in Lemma \ref{jieduan11}, there exists $C>0$ independent of $r$ such that
\[\left|\int_{|\eta|\le r\delta}
O\left(1\right)\chi\left(\frac{\eta}{r\delta}\right)\,
e^{\,i\,\omega\cdot\eta}\,
d\eta\right|\le C.\]
Therefore, 
\[
\lim\limits_{|k|\rightarrow \infty}I(k)\;\longrightarrow\;
\lim\limits_{N\rightarrow \infty}
\int_{\mathbb{R}^d}|\eta|^{-2t}\chi\left(\frac{\eta}{N}\right)\,e^{i\omega\cdot\eta}\,d\eta,
\]
which is the desired result by Lemma \ref{jieduan11}.\qed

\section{Appendix: Proof of Lemma \ref{jieduan11}}

In this appendix, we give the detailed proof of Lemma \ref{jieduan11}.

\begin{proof}
	We first carry out the proof for dimensions $d \ge 2.$
	
	\textbf{\emph{(1) Existence and finiteness.}} 
	Write $\eta=r\,\theta$ with $r=|\eta|\in[0,\infty)$ and $\theta\in\mathbb{S}^{d-1}$, and denote by $d\sigma(\theta)$ the surface measure.  Then
	\[
	\int_{\mathbb{R}^d}|\eta|^{2s}\,\chi\!\Bigl(\tfrac\eta N\Bigr)\,e^{i\omega\cdot\eta}\,d\eta
	=\int_{0}^{\infty}r^{2s+d-1}\,\chi\!\Bigl(\tfrac rN\Bigr)
	\Bigl[\int_{\mathbb{S}^{d-1}}e^{\,i r(\omega\cdot\theta)}\,d\sigma(\theta)\Bigr]\,dr.
	\]
	By the classical spherical‐average identity (see \cite[B.4]{grafakos2008classical}),
	\[
	\int_{\mathbb{S}^{d-1}}e^{\,i r(\omega\cdot\theta)}\,d\sigma(\theta)
	=(2\pi)^{\frac d2}\,r^{-\frac d2+1}\,J_{\frac d2-1}(r),\:d\ge 2
	\]
	where $J_\nu$ is the Bessel function of the first kind.  Hence
	\[
	I_N
	:=\int_{\mathbb{R}^d}|\eta|^{2s}\,\chi\!\Bigl(\tfrac\eta N\Bigr)\,e^{i\omega\cdot\eta}\,d\eta
	=(2\pi)^{\frac d2}
	\int_{0}^{\infty}r^{2s+\frac d2}\,\chi\!\Bigl(\tfrac rN\Bigr)\,J_{\frac d2-1}(r)\,dr.
	\]
	We integrate by parts using the relation
	\[
	\frac{d}{dr}\bigl[r^{\nu+1}J_{\nu+1}(r)\bigr]
	=r^{\nu+1}J_{\nu}(r),
	\quad
	\nu=\frac d2-1.
	\]
	Writing $K(r)=r^{2s}\chi(r/N)$, one finds
	\[
	\int_{0}^\infty K(r)r^{\nu+1}J_{\nu}(r)\,dr
	=\bigl[K(r)\,(r^{\nu+1}J_{\nu+1}(r))\bigr]_{0}^{\infty}
	+\int_{0}^\infty K'(r)\,r^{\nu+1}J_{\nu+1}(r)\,dr.
	\]
	Recall the standard asymptotics of the Bessel function (see e.g.\cite[p.223]{olver2010nist}): 
	\begin{equation}\label{bessel}
		J_{\nu}(r)\sim r^{\nu}
		\quad r\to0,
		\qquad
		J_{\nu}(r)\lesssim O\bigl(r^{-1/2}\bigr)
		\quad r\to\infty.
	\end{equation}
	Thus, the boundary term vanishes and 
	\[
	\int_{0}^\infty K(r)r^{\nu+1}J_{\nu}(r)\,dr
	=\int_{0}^\infty\left(2sr^{2s-1}\chi\left(\frac{r}{N}\right)+\frac{r^{2s}}{N}\chi^{\prime}\left(\frac{r}{N}\right)\right)\,r^{\frac{d}{2}}J_{\frac{d}{2}}(r)\,dr.	\]
	Repeat $k$ times where $k$ is sufficiently large, we obtain that
	\[		\int_{0}^\infty K(r)r^{\nu+1}J_{\nu}(r)\,dr\sim \sum_{i=0}^{k}\frac{1}{N^{k-i}}\int_0^{\infty}r^{2s-i}\chi^{(k-i)}\left(\frac{r}{N}\right)r^{\frac{d}{2}}J_{\frac{d}{2}+k-1}(r)\,dr.\]
	By \cite[p.243]{olver2010nist},
	\[
	\int_{0}^{\infty} r^{\mu} \, J_{\nu}(r)\,dr
	\;=\;
	2^{\mu}\,
	\frac{\Gamma\!\bigl(\tfrac12\nu+\tfrac12\mu+\tfrac12\bigr)}
	{\Gamma\!\bigl(\tfrac12\nu-\tfrac12\mu+\tfrac12\bigr)},
	\quad
	\Re(\mu+\nu)>-1,\;\Re\mu<\tfrac12.
	\]
	Thus, for $2s+\frac{d-1}{2}<i< k$, we obtain that
	\begin{equation}\label{fuza1}
		\lim\limits_{N\rightarrow \infty}\frac{1}{N^{k-i}}\int_0^{\infty}r^{2s-i}\chi^{(k-i)}\left(\frac{r}{N}\right)r^{\frac{d}{2}}J_{\frac{d}{2}+k-1}(r)\,dr=0
	\end{equation}
	and for $i=k$, by the dominated convergence theorem, we have
	\begin{equation}\label{fuza2}
		\lim\limits_{N\rightarrow \infty}\int_0^{\infty}r^{2s-k}\chi\left(\frac{r}{N}\right)r^{\frac{d}{2}}J_{\frac{d}{2}+k-1}(r)\,dr=\int_0^{\infty}r^{2s-k+\frac{d}{2}}J_{\frac{d}{2}+k-1}(r)\,dr.
	\end{equation}
	For $0 \le i \le 2s+\frac{d-1}{2},$ then for sufficiently large $k$, $k-i\ge k-2d,$ by (\ref{bessel})
	\[\left|\int_0^{\infty}r^{2s-i}\chi^{(k-i)}\left(\frac{r}{N}\right)r^{\frac{d}{2}}J_{\frac{d}{2}+k-1}(r)\,dr\right|\lesssim 1+N^{2s+\frac{d+1}{2}},\quad \text{sufficiently large $N$.}\]
	Thus, if we take $k>6d$, then for  $0 \le i \le 2s+\frac{d-1}{2},$
	\begin{equation}\label{fuza3}
		\lim\limits_{N\rightarrow \infty}\frac{1}{N^{k-i}}\int_0^{\infty}r^{2s-i}\chi^{(k-i)}\left(\frac{r}{N}\right)r^{\frac{d}{2}}J_{\frac{d}{2}+k-1}(r)\,dr=0.
	\end{equation}
	Combining the results of (\ref{fuza1}), (\ref{fuza2}), and (\ref{fuza3}), we deduce that the limit
	\[J_\chi(\omega)=\lim_{N\to\infty}
	\int_{\mathbb{R}^d}|\eta|^{2s}\,\chi\!\Bigl(\frac{\eta}{N}\Bigr)\,e^{\,i\,\omega\cdot\eta}\,d\eta\]
	exists and is finite.
	
	When $d = 1,$ the existence of the limit can be verified using integration by parts together with Lemma 2 in \cite[Chapter V]{stein1970singular}. We omit the details.

	\textbf{\emph{(2) Independence of the cutoff.}} In fact, it can be seen from the proof of the existence of the limit in the first step that  the limit is independent of the choice of the cutoff function.
	
	\textbf{\emph{(3) Independence of the direction.}}  Let \(\omega,\omega'\in\mathbb S^{d-1}\).  Choose an orthogonal matrix such that \(R\omega=\omega'\).  Writing \(\eta=R\zeta\) and using \(|\eta|=|\zeta|\) gives
	\[
	\int_{\mathbb{R}^d}|\eta|^{2s}\,\chi\!\Bigl(\tfrac\eta N\Bigr)\,e^{\,i\,\omega\cdot\eta}\,d\eta
	=\int_{\mathbb{R}^d}|\zeta|^{2s}\,\chi\!\Bigl(\tfrac{R\zeta}{N}\Bigr)\,e^{\,i\,\omega'\cdot\zeta}\,d\zeta.
	\]
	Since \(\chi\) is radial, \(\chi(R\zeta/N)=\chi(\zeta/N)\), so the two integrals agree for each \(N\).  Therefore their common limit
	\(\displaystyle J_\chi(\omega)=J_\chi(\omega')\),
	showing that \(J_\chi\) is independent of direction.
\end{proof}

\section{Acknowledgements}

We thank Zhiqin Lu for his foundational insight into the functional calculus approach to the logarithmic Laplacian, Bobo Hua for his careful reading of the manuscript and constructive suggestions, and Huyuan Chen for his detailed feedback on various theoretical aspects of the logarithmic Laplacian. We sincerely thank Zuoshunhua Shi for the key idea regarding fractional spatial frequency decay, Danglin Li for valuable discussions on this topic, and Jiaxuan Wang for his insights into the asymptotic behavior of heat kernels.

	\bibliographystyle{unsrt}
	\bibliography{refs}

\end{document}